\documentclass[english,12pt]{article}

\usepackage[utf8]{inputenc}
\usepackage[T1]{fontenc}
\usepackage[a4paper,margin=2.5cm]{geometry}
\usepackage{amssymb,amsthm,mathtools,mathrsfs}
\usepackage{graphicx}
\usepackage{babel}
\usepackage[colorlinks]{hyperref}
\usepackage{microtype}
\usepackage{caption}
\usepackage{subcaption}
\usepackage{enumitem}

\usepackage{tikz}
\usetikzlibrary{cd}
\tikzcdset{scale cd/.style={every label/.append style={scale=#1}, cells={nodes={scale=#1}}}}

\usepackage{bm}
\usepackage{siunitx}
\usepackage[capitalize,nameinlink]{cleveref}
\numberwithin{equation}{section}
\allowdisplaybreaks{}

\usepackage[backend=biber,style=numeric,sorting=nyt]{biblatex}
\usepackage{csquotes}

\definecolor{Blue}{HTML}{1F77B4}
\definecolor{Orange}{HTML}{FF7F0E}
\definecolor{Green}{HTML}{2CA02C}
\definecolor{Red}{HTML}{D62728}
\definecolor{Grey}{HTML}{7F7F7F}

\hypersetup{colorlinks = True,
	linkcolor  = Red,
	citecolor  = Green,
	urlcolor   = Blue,
}

\NewDocumentCommand{\bbC}{}{\mathbb{C}}

\NewDocumentCommand{\bbL}{}{\mathbb{L}}
\NewDocumentCommand{\bbN}{}{\mathbb{N}}
\NewDocumentCommand{\bbR}{}{\mathbb{R}}

\NewDocumentCommand{\calK}{}{\mathcal{K}}
\NewDocumentCommand{\calL}{}{\mathcal{L}}
\NewDocumentCommand{\calN}{}{\mathcal{N}}
\NewDocumentCommand{\calP}{}{\mathcal{P}}

\NewDocumentCommand{\calX}{}{\mathcal{X}}
\NewDocumentCommand{\calY}{}{\mathcal{Y}}

\NewDocumentCommand{\scrE}{}{\mathscr{E}}

\NewDocumentCommand{\ex}{}{\mathsf{e}}
\NewDocumentCommand{\im}{}{\mathsf{i}\mkern1mu}

\NewDocumentCommand{\msH}{}{\mathcal{H}}

\NewDocumentCommand{\spC}{}{\mathscr{C}}
\NewDocumentCommand{\spL}{}{\mathrm{L}}
\NewDocumentCommand{\spM}{}{\mathcal{M}}

\NewDocumentCommand{\spPolyP}{o}{\IfNoValueTF{#1}{\mathbb{P}}{\mathbb{P}_{#1}}}
\NewDocumentCommand{\spPolyHom}{m}{\mathbb{H}_{#1}}
\NewDocumentCommand{\spPolyQ}{m}{\mathbb{Q}_{#1}}

\DeclareMathOperator{\Id}{\mathsf{Id}}
\DeclareMathOperator{\D}{D}
\NewDocumentCommand{\di}{m}{\mathop{}\!\mathrm{d}#1}
\DeclareMathOperator{\OO}{\mathcal{O}}

\DeclareMathOperator{\Dim}{dim}

\DeclareMathOperator{\vect}{span}
\DeclareMathOperator{\Hom}{\mathcal{L}}
\DeclareMathOperator{\diag}{diag}
\DeclareMathOperator{\Lip}{Lip}
\DeclareMathOperator{\Img}{Im}
\DeclareMathOperator{\Ker}{Ker}
\DeclareMathOperator{\Supp}{supp}
\DeclareMathOperator{\Diam}{diam}
\DeclareMathOperator{\Span}{span}

\DeclarePairedDelimiter{\plr}\lparen\rparen
\DeclarePairedDelimiter{\clr}\lbrack\rbrack
\DeclarePairedDelimiter{\blr}\lbrace\rbrace

\DeclarePairedDelimiter{\abs}\lvert\rvert
\DeclarePairedDelimiter{\norm}\lVert\rVert

\DeclarePairedDelimiterX\setst[2]\lbrace\rbrace{#1\:\delimsize\vert\:\mathopen{}#2} \DeclarePairedDelimiterX\setwt[2]\lbrace\rbrace{#1\:{:}\:\mathopen{}#2}

\DeclarePairedDelimiterX\ioo[2]\lparen\rparen{#1,\:#2}
\DeclarePairedDelimiterX\ioc[2]\lparen\rbrack{#1,\:#2}
\DeclarePairedDelimiterX\ico[2]\lbrack\rparen{#1,\:#2}
\DeclarePairedDelimiterX\icc[2]\lbrack\rbrack{#1,\:#2}

\NewDocumentCommand\restr{s m m}{\IfBooleanTF{#1}{\left. #2 \right\vert_{#3}}{#2 \arrowvert_{#3}}}

\RenewDocumentCommand{\vec}{m}{\bm{#1}}
\NewDocumentCommand{\eps}{}{\varepsilon}

\NewDocumentCommand{\matI}{}{I}
\NewDocumentCommand{\trans}{}{\intercal}

\NewDocumentCommand{\IFS}{}{\mathscr{S}}
\DeclareMathOperator{\OpSet}{\mathscr{H}}
\DeclareMathOperator{\OpFct}{\mathscr{F}}
\DeclareMathOperator{\OpMsr}{\mathscr{M}}
\DeclareMathOperator{\Quad}{\mathcal{Q}}
\DeclareMathOperator{\Eval}{\mathcal{E}}
\DeclareMathOperator{\Lpoly}{\mathscr{L}}
\DeclareMathOperator{\Interp}{\mathcal{I}}

\theoremstyle{plain}
\newtheorem{definition}{Definition}[section]
\newtheorem{lemma}[definition]{Lemma}
\newtheorem{proposition}[definition]{Proposition}
\newtheorem{theorem}[definition]{Theorem}
\newtheorem{corollary}[definition]{Corollary}
\newtheorem{remark}[definition]{Remark}
\newtheorem{assumption}[definition]{Assumption}

\theoremstyle{remark}
\newtheorem{example}[definition]{Example}

\crefname{assumption}{Assumption}{Assumptions}
\Crefname{assumption}{Assumption}{Assumptions}

\title{High-order numerical integration on self-affine sets.}

\author{%
    Patrick Joly\thanks{POEMS, CNRS, Inria, ENSTA, Institut Polytechnique de Paris, 91120 Palaiseau, France}
	\and
	Maryna Kachanovska\footnotemark[1]
	\and
	Zo{\"\i}s Moitier\thanks{Inria, Unité de Mathématiques Appliquées, ENSTA, Institut Polytechnique de Paris, 91120 Palaiseau, France}
}

\begin{document}

\maketitle

\begin{abstract}
	We construct an interpolatory high-order cubature rule to compute integrals of smooth functions over self-affine sets with respect to an invariant measure.
	The main difficulty is the computation of the cubature weights, which we characterize algebraically, by exploiting a self-similarity property of the integral.
	We propose an \( h \)-version and a \( p \)-version of the cubature, present an error analysis and conduct numerical experiments.
\end{abstract}

\section{Introduction}\label{sec:intro}

A significant portion of research in modern numerical analysis is dedicated to numerical approximation of solutions to the problems posed on rough domains or of irregular functions, see~\cite{achdou06,banjai07,capitanelli10,lancia13,banjai21,cefalo21,dekkers221,dekkers222,roesler24} and references therein.
In particular, a series of recent articles~\cite{cw16,cw18,cw21,cw22,caetano2024,cw25} deals with approximation of wave scattering by fractal screens, a problem that has applications in antenna engineering.
The underlying numerical method, suggested in~\cite{caetano2024}, relies on approximating integrals on fractal screens. Very few works have so far addressed the question of constructing cubature on fractal sets, which is a goal of the present work.

It seems that one of the first works treating cubature on fractal sets is due to G.~Mantica~\cite{mantica96}, who constructs a Gaussian quadrature for 1D fractals based on orthogonal polynomials.
The chaos-game cubatures, based on Monte-Carlo method, were suggested in~\cite{Forte1998}.
In~\cite{gibbs2023}, the authors design and analyze a barycentric rule for integration of regular functions on fractals, and propose a clever method in~\cite{gibbs2024} to compute singular integrals of BEM\@.
Up to our knowledge, existing methods are either uni-dimensional, or of low order.

Our approach is inspired by the work of Strichartz~\cite{Str2000}, who suggests an elegant method to evaluate integrals of polynomials on fractals using self-similarity (see \cref{appendix:gen}).
Adapting this idea to our setting yields a cubature method of arbitrary high order for integrating functions on self-similar sets in \( \bbR^n \).

This article is organized as follows.
In \cref{sec:setting} we introduce the notions related to the geometry and measure of self-similar sets, as well as polynomial spaces.
\Cref{sec:cubature-S-inv} and \cref{sec:cubature-non-S-inv} describe the new method to construct cubature on self-affine sets.
In \cref{sec:error} we propose two versions of the cubature: an \( h \)-version and a \( p \)-version, and discuss the related error estimates.
Finally, \cref{sec:numeric} contains numerical experiments.

\section{Problem setting}\label{sec:setting}

\subsection{Geometry, measure and integration}

\subsubsection{Self-affine sets}

The exposition below follows~\cite[Ch.~9]{Fal2014}.
We will denote by  \( \abs{\: \cdot \:}_2 \) the Euclidean norm on \( \bbR^n \).

\begin{definition}\label{def:ifs}
	An \emph{Iterated Function System} (IFS) on \( \bbR^n \) is a finite family of contractive maps \( \IFS = \setwt{S_\ell \colon \bbR^n \to \bbR^n}{\ell= 1, 2, \ldots, L} \), with \( L \geq 2 \).
	Namely, for \( \ell \in \bbL \coloneqq \blr{1, 2, \ldots, L} \), there exists \( 0 \leq \rho_\ell < 1 \) such that the maps \( S_\ell \) satisfy the estimation \( \abs{S_\ell(x) - S_\ell(y)}_2 \leq \rho_\ell \, \abs{x - y}_2 \), for all \( x, y \in \bbR^n \).
	The fixed point of a contractive mapping \( S_\ell \) is  denoted by \( c_{\ell} \).
\end{definition}

\begin{definition}\label{def:hut-op}
	Let \( \calK \) denote the set of all nonempty compact sets of \( \bbR^n \).
	For an IFS \( \IFS = \setwt{S_\ell}{\ell \in \bbL} \), we define the \emph{Hutchinson operator} by
	\[
		\OpSet \colon \calK \to \calK
		\quad \text{such that} \quad
		\OpSet(E) \coloneqq S_1\plr{E} \cup \cdots \cup S_{L}\plr{E}.
	\]
\end{definition}

The above definitions are motivated by the following result from~\cite[Thm.~9.1]{Fal2014}.

\begin{theorem}\label{thm:attractor-existence}
	For an IFS \( \IFS = \setwt{S_\ell}{\ell \in \bbL} \), there exists a unique nonempty compact set \( \Gamma \in \calK \), which is a fixed point of \( \OpSet \), meaning that \( \Gamma = \OpSet\plr{\Gamma} \). This set is called the \emph{attractor} of the IFS \( \IFS \).
\end{theorem}

An important class of IFS are the \emph{affine IFS}, for which all the maps \( S_{\ell} \) are affine:
\begin{equation}\label{eq:ifs-affine}
	S_\ell \colon x \mapsto A_\ell x + b_\ell
	\quad \text{and} \quad
	c_\ell = \plr*{\matI - A_\ell}^{-1} b_\ell.
\end{equation}
Then the contraction \( \rho_\ell \) is equal to the spectral norm of the matrix \( A_\ell \).
A classic subclass of affine IFS are the \emph{similar IFS} where the maps \( S_\ell \) are contractive similarities, meaning that the maps \( S_\ell \) satisfy \( \abs{S_\ell(x) - S_\ell(y)}_2 = \rho_\ell \, \abs{x - y}_2 \), for all \( x, y \in \bbR^n \) which is equivalent to saying that
\begin{equation}\label{eq:ifs-contractive-similitude}
	A_\ell = \rho_\ell \, T_\ell,
	\qquad \text{where} \ T_\ell \ \text{is an orthogonal matrix}.
\end{equation}

In what follows, we will refer to an attractor of an IFS\@ in \( \bbR^n \) as a \emph{fractal set}.
For the particular case when the IFS is affine (resp.~similar), the corresponding attractor will be referred to as a \emph{self-affine set} (resp.~\emph{self-similar set}).
From the above definition, we have immediately the following result, which will be of importance later.

\begin{lemma}\label{lem:kI}
	The pre-image of \( \Gamma \) by the maps \( S_\ell \), \( \ell\in \bbL \), satisfies \( S_\ell^{-1}(\Gamma) \cap \Gamma = \Gamma \).
\end{lemma}
\begin{proof}
	Since \( \OpSet(\Gamma) = \Gamma \), we have, with \( S_\ell^{-1}\plr*{\Gamma} \) denoting the pre-image of \( \Gamma \),
	\[
		S_\ell^{-1}\plr*{\Gamma}
		= S_\ell^{-1}\plr*{\OpSet(\Gamma)}
		= S_\ell^{-1}\plr*{S_1(\Gamma)} \cup \cdots \cup S_\ell^{-1}\plr*{S_{L-1}(\Gamma)}
		\supset S_\ell^{-1}\plr*{S_\ell(\Gamma)}.
	\]
	The desired result follows from \( \Gamma \subset S_\ell^{-1}\plr{S_\ell(\Gamma)} \).
\end{proof}

In general, fractal sets are non-Lipschitz, but can be approximated by sets of a simpler structure, called \emph{pre-fractals} or \emph{pre-attractors}.
A sequence of pre-fractals can be constructed as \( \OpSet^p\plr{F} \), with \( F \) being an arbitrary compact set.
In a well-chosen topology, see the proof of \cref{thm:attractor-existence} in~\cite[Thm.~9.1]{Fal2014}, \( \OpSet^p\plr{F} \to \Gamma \) as \( p \to +\infty \).

\begin{example}\label[example]{ex:sas}
	In \cref{fig:sas-examples}, we have plotted classic examples of self-affine sets.
	The IFS corresponding  to the Fat Sierpi\'nski triangle, see~\cite{BMS2004,Jor2006}, is composed of the three maps \( x \to \rho x + (1-\rho) c_\ell \) where \( \rho = \plr{\sqrt{5}-1}/2 \) and the \( c_\ell \) are the vertices of an equilateral triangle.
	The one for the Koch snowflake is composed of seven contractive similitudes, see~\cite[Fig.~3]{gibbs2023}, the IFS for the Barnsley fern is composed of four maps where the coefficients are given by~\cite[Tbl.~3.8.3]{barnsley}, and a non-symmetric Cantor dust\footnote{\( A_\ell = \rho_\ell R_{\theta_\ell} \) and \( b_\ell = (1-\rho_\ell) c_\ell \) with \( (\rho_\ell, \theta_\ell, c_\ell) = \plr{0.25, 0.4, (-1.4, -1.1)} \), \( \plr{0.35, 0.2, (0.8, -0.7)} \), \( \plr{0.3, 0.3, (1.2, 1.3)} \), \( \plr{0.4, 0.1, (-1.3, 0.9)} \).}
	\begin{figure}[tbhp]
		\centering
		\subfloat[]{\label{fig:sierpinski-triangle-fat}
			\includegraphics[height=3.05cm]{{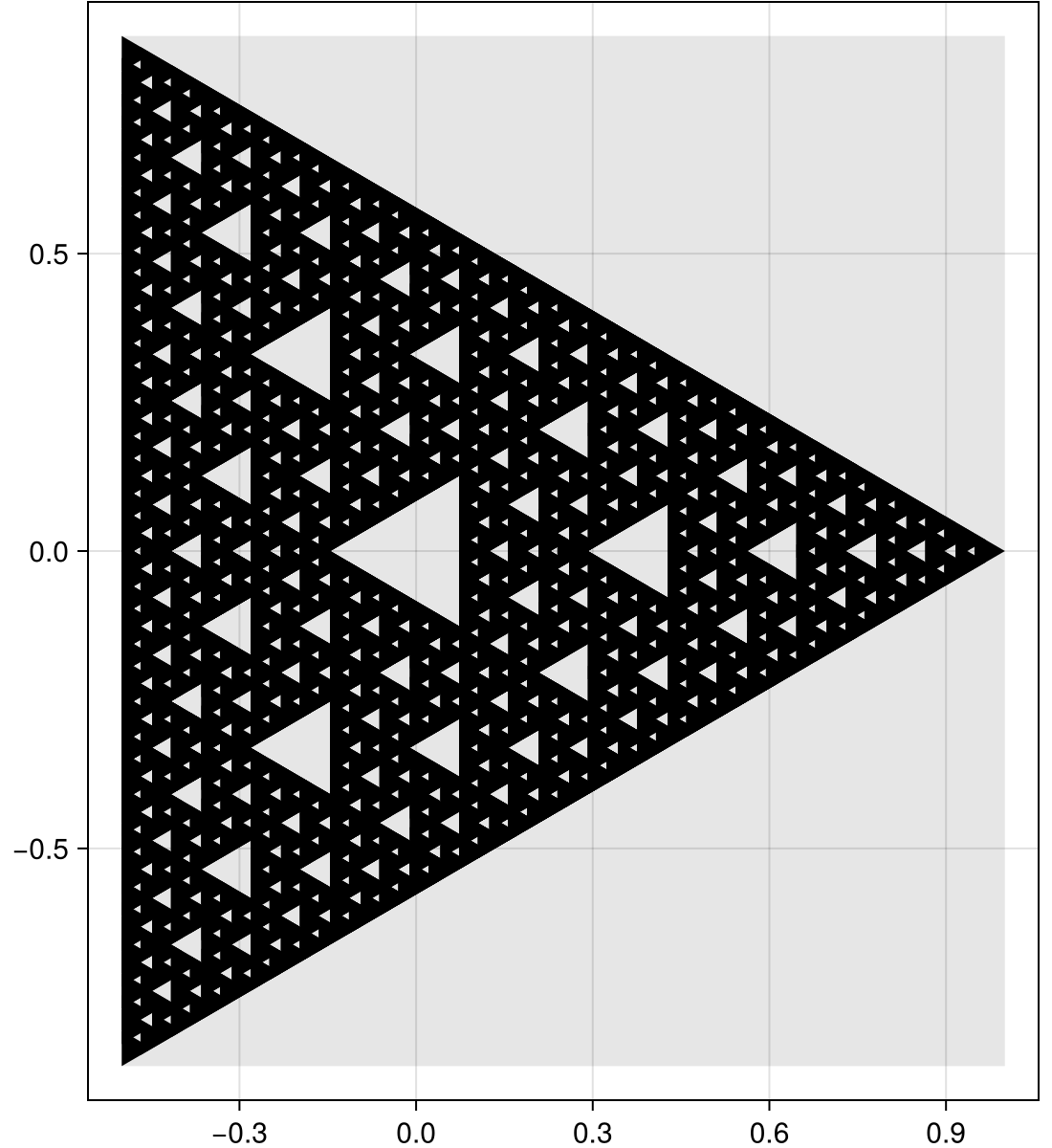}}
		}\subfloat[]{\label{fig:koch-snowflake}
			\includegraphics[height=3.05cm]{{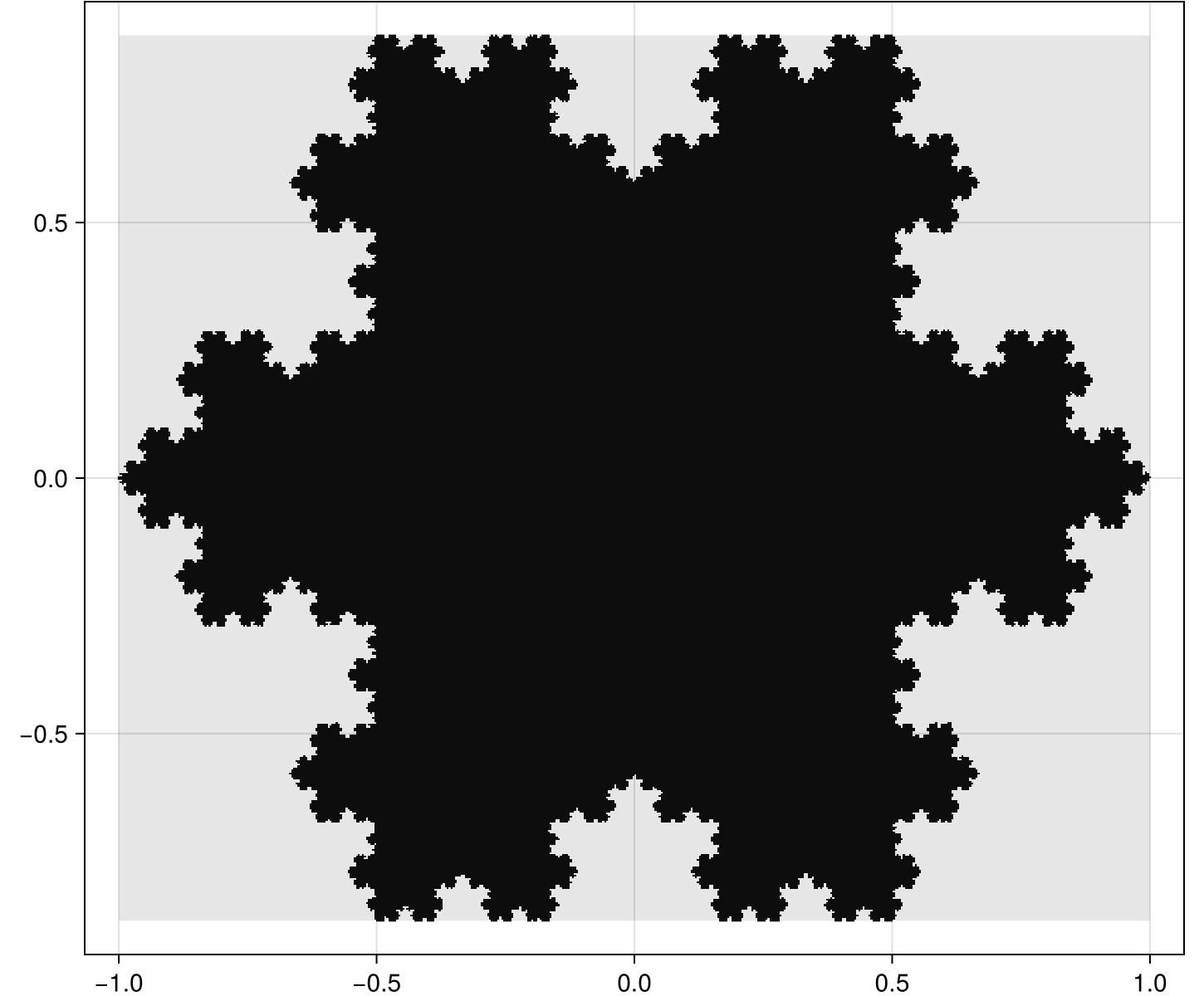}}
		}\subfloat[]{\label{fig:barnsley-fern}
			\includegraphics[height=3.05cm]{{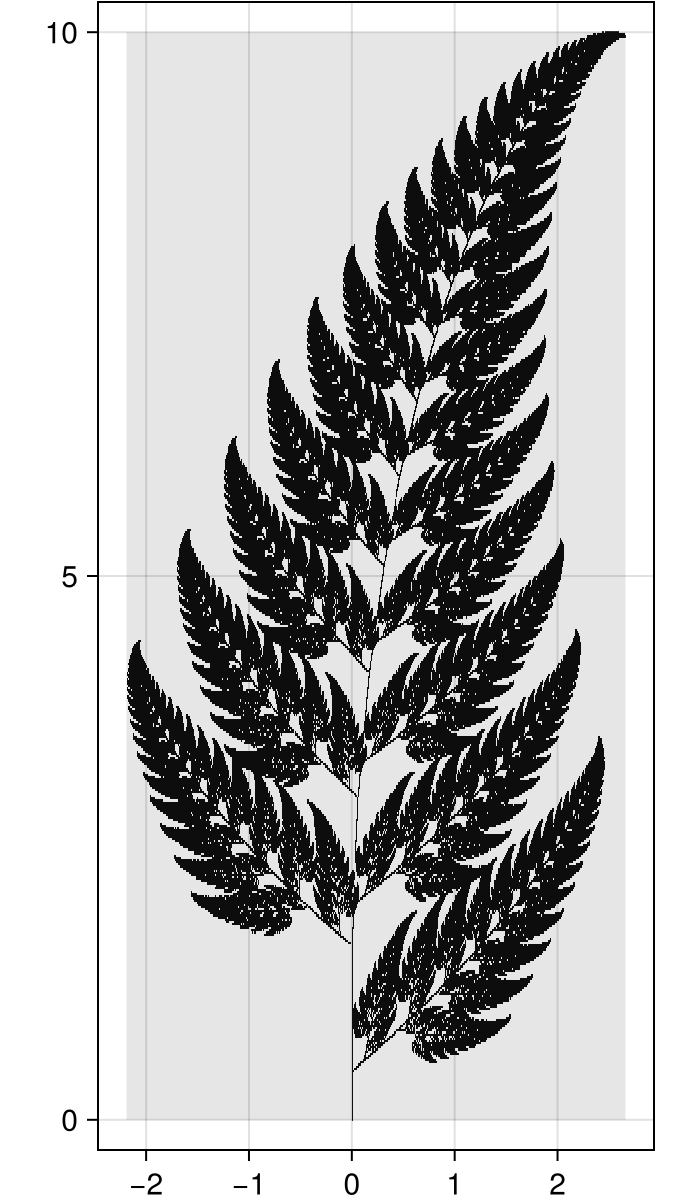}}
		}\subfloat[]{\label{fig:cantor-non-sym.pdf}
			\includegraphics[height=3.05cm]{{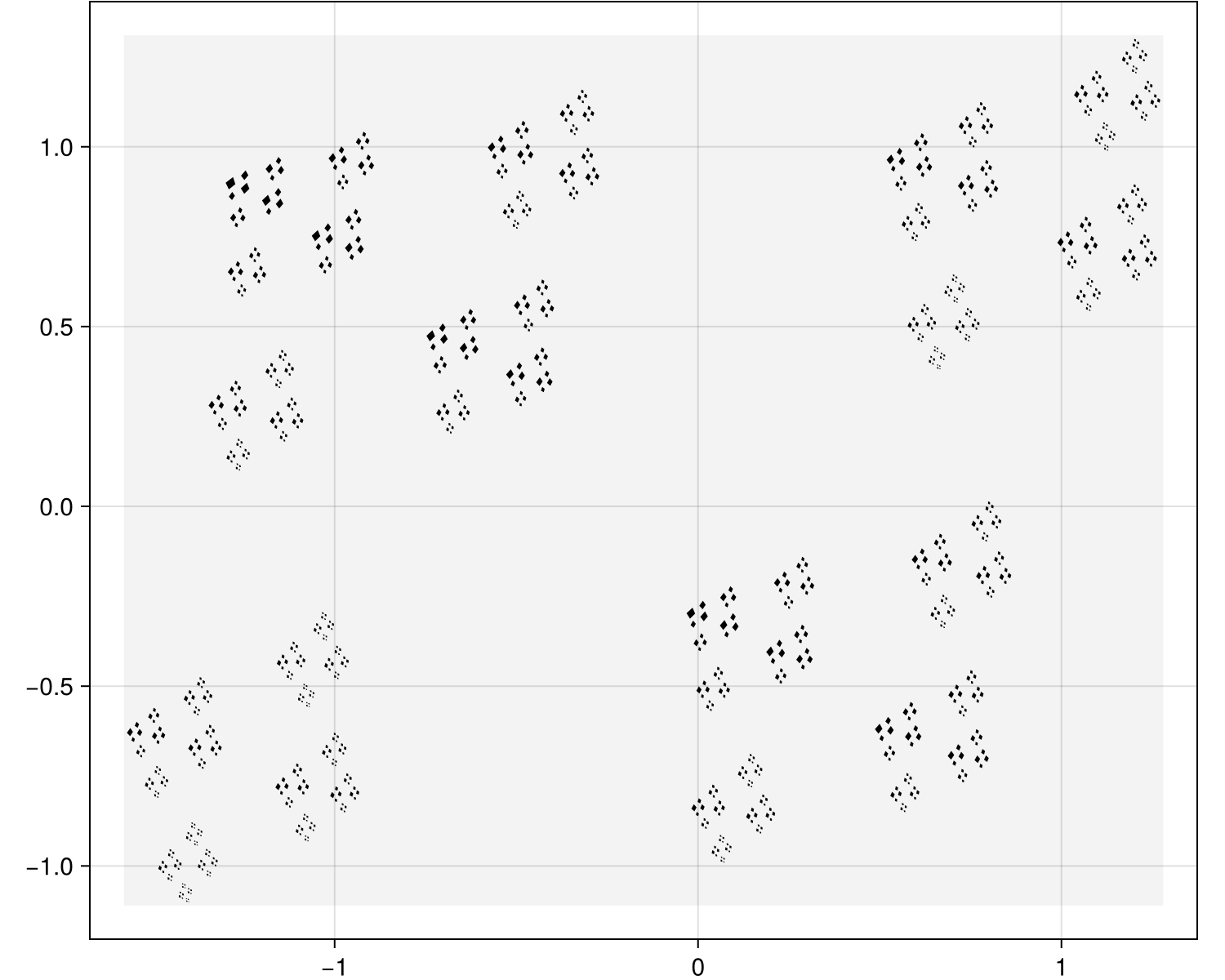}}
		}\caption{Examples of self-affine sets: (a) the fat Sierpi\'nski triangle, (b) the Koch snowflake, (c) the Barnsley fern (c), and a non-symmetric Cantor dust.
			The shaded box is the approximate best fitting rectangle, see \cref{sec:numeric}.}\label{fig:sas-examples}
	\end{figure}
\end{example}

\subsubsection{Invariant measure}\label{sec:inv_measure}

The exposition on invariant measure of an IFS is taken from~\cite{Hut1981}.
Given a nonempty compact subset \( F \subset \bbR^n \), equipped with the usual Euclidean topology, one can define an \emph{outer measure} on \( F \) as a map \( \nu \colon 2^F \to \icc{0}{+\infty} \) such that \( \nu(\varnothing) = 0 \) (the \emph{null empty set} condition) and, for any \( A, B_0, B_1, \ldots \subset F \), \( A \subset \bigcup_{i \in \bbN} B_i \) implies \( \nu\plr{A} \leq \sum_{i \in \bbN} \nu\plr{B_i} \) (the \emph{countably subadditive} condition).
We say that a set \( A \subset F \) is \( \nu \)-\emph{measurable},
if \( \nu(E) = \nu(E \cap A) + \nu(E \setminus A) \) for all \( E \subset F \) (the Carath\'eodory's criterion).
The family of measurable sets forms a \( \sigma \)-algebra.
When restricted to \( \nu \)-measurable sets, the outer measure \( \nu \) produces a measure.
In the sequel, for a reason that will appear in \cref{thm:inv_measure}, we shall restrict ourselves to \emph{Borel regular} measures: an outer measure \( \nu \) is \emph{Borel regular} if, and only if, all Borel sets\footnote{The Borel sets are the elements of the smallest \( \sigma \)-algebra containing all the open sets of \( F \).}
are measurable, and, for all \( A \subset F \) there exists a Borel set \( B \), such that \( A \subset B \) and \( \nu(A) = \nu(B) \).
The support of the measure \( \nu \) is defined by \( \Supp\nu \coloneqq F \setminus \bigcup \setst{V \ \text{open}}{\nu(V) = 0} \).
Let us further introduce
\[
	\spM_F^1 \coloneqq \setst*{\nu}{\nu \text{ is Borel regular and} \ \nu(F) = 1}.
\]
Given \( g \colon F \to F \), measurable, the push-forward measure \( \nu \circ g^{-1} \in \spM_F^1 \) is defined by
\begin{align}\label{eq:push_forward}
	\forall A \subset F, \quad \nu \circ g^{-1}\plr{A} \coloneqq \nu\plr{g^{-1}(A)}, \text{ with }g^{-1}(A) \coloneqq \setst{x \in F}{g(x) \in A}.
\end{align}

Let us now construct an invariant measure associated to an affine IFS \( \IFS \) given by the maps \( \setwt{S_\ell}{\ell \in \bbL} \).
First, let \( F \) satisfy
\( \OpSet\plr{F} \subset F \), which is true if \( F \) is chosen a closed ball of large enough radius.
Given \( \vec{\mu} = \plr{\mu_1, \ldots, \mu_L} \in \ioo{0}{1}^L \) such that \( \abs{\vec{\mu}}_1 \coloneqq \sum_{\ell\in \bbL} \mu_{\ell} = 1 \), let us define the map \( \OpMsr \colon \spM_F^1 \to \spM_F^1 \) via the following identity\footnote{The interested reader can verify that \( \OpMsr \) maps \( \spM_F^1 \) into itself because \( \OpSet\plr{F} \subset F \) and \( \abs{\vec{\mu}}_1 = 1 \).}:
\begin{equation}\label{eq:def_op_msr}
	\OpMsr \colon \nu \longmapsto \sum_{\ell \in \bbL} \, \mu_\ell \ \nu \circ S_\ell^{-1}.
\end{equation}
A measure satisfying \( \nu = \OpMsr\nu \) is called \emph{invariant} with respect to \( \plr{\IFS, \vec{\mu}} \).

\begin{theorem}\label{thm:inv_measure}
	Let \( \IFS = \setwt{S_\ell}{\ell \in \bbL} \) and \( \vec{\mu} \), \( \OpMsr \) be as in \cref{eq:def_op_msr}.
	Then there exists a unique invariant measure \( \mu \in \spM_F^1 \).
	Furthermore, the support of the measure \( \mu \) is \( \Gamma \), which implies that for any measurable \( A \subset F \), we have \( \mu(A)=\mu(A\cap \Gamma) \).
\end{theorem}
\begin{proof}
	This proof is due to~\cite{Hut1981}; here we fill in some arguments omitted in the above reference.
	We equip \( \spM_F^1 \) with the Monge-Kantorovich distance \( d(\mu, \nu) \):
	\[
		\forall \plr{\mu, \nu} \in \spM_F^1, \quad
		d(\mu, \nu) = \sup \setst*{\int_F f(x) \di{\mu} - \int_F f(x) \di{\nu}}{f \colon F \to \bbR,\ \Lip(f) \leq 1}.
	\]
	Then \( (\spM_F^1, d) \) is a complete metric space, see the discussion after~\cite[Thm.~8.10.43, Vol.~II]{Bog2007}; in the latter reference, the author restricts his attention to the Radon measures, however, since \( F \) is complete and separable, measures in \( \spM_F^1 \) are Radon measures, see~\cite[Thm.~7.1.7, Vol.~II]{Bog2007}.
	The operator \( \OpMsr \) is a contractive map on this metric space, see the proof of~\cite[Thm.~4.4(1)]{Hut1981}.  Existence and uniqueness of \( \mu \) follows by the fixed point theorem.
	The property \( \Supp\mu = \Gamma \) is proven in~\cite[Thm.~4.4(4)]{Hut1981}.
\end{proof}

\begin{remark}[Hausdorff measure and application to the wave scattering by fractal screens]\label[remark]{rem:hausdorff} Given an IFS, one can construct infinitely many invariant measures, parame\-trized by \( \vec{\mu} \).
	Nonetheless, in some cases, this measure coincides with the (normalized) \( d \)-dimensional Hausdorff measure \( \msH^d(\Gamma) \) (see~\cite[3.10(iii)]{Bog2007}).
	Assume that the IFS \( \IFS = \setwt{S_\ell}{\ell \in \bbL} \) satisfy the \emph{Open Set Condition} (OSC).
	More precisely, there exists a nonempty bounded open set \( U \) such that \( \OpSet\plr{U} \subset U \) and \( S_i(U) \cap S_j(U) = \varnothing \) for \( i \neq j \in \bbL \).
	Moreover, let \( \IFS \) be a similar IFS, \emph{cf.}~\cref{eq:ifs-contractive-similitude}.
	In this case one can show that the Hausdorff dimension \( d \) of \( \Gamma \) is a unique solution to \( \sum_{\ell \in \bbL} \rho_\ell^d = 1 \), see~\cite[Thm.~9.3]{Fal2014}.
	Choosing \( \mu_\ell = \rho_\ell^d \) yields the invariant measure \( \mu \colon E \mapsto \msH^d\plr{E \cap \Gamma} / \msH^d(\Gamma) \), see~\cite[Thm.~5.3(1)(iii)]{Hut1981}.
	This example is of importance for applications in the computational scattering theory.
	Recently, Hausdorff measure based boundary element methods (BEM) were developed for scattering by fractal screens, see~\cite{caetano2024}, and extended to non-planar fractals in~\cite{cw25}.
	A non-exhaustive list of self-similar sets satisfying the OSC includes Sierpi\'nski gasket, Koch snowflake and Cantor sets.
	Let us remark that one of favorable properties of the Hausdorff measure compared
	to a generic invariant measure is its invariance with respect to orthogonal transformations.
	This can be used for efficient computation of singular integrals arising in the Hausdorff BEM, \emph{cf.}~\cite{gibbs2023,gibbs2024}.
\end{remark}

\subsubsection{Integrals with respect to invariant measures}

Let us fix \( \vec{\mu}\in \ioo{0}{1}^L \), with \( \abs{\vec{\mu}}_1 = 1 \) and \( \mu \) the associated invariant measure as defined in \cref{thm:inv_measure}.

As discussed in the introduction, the goal of this work is to approximate numerically \( \int_{\Gamma}f \di{\mu} \), with \( f \) being sufficiently regular and defined in the vicinity of \( \Gamma \).
This will be done by using a certain self-similarity property of the integral, which is introduced in the present section.
The proof of this property is in turn based on the following change-of-variables formula, see~\cite[Thm.~3.6.1.]{Bog2007}:
\begin{equation}\label{eq:change-of-variable}
	\int_{\Gamma} f(x) \di\plr{\mu \circ g^{-1}} = \int_{g^{-1}(\Gamma)} f \circ g(x) \di{\mu},
\end{equation}
In the above, \( g \colon \Gamma \rightarrow \Gamma \) is measurable, and the measurable function \( f \) is such that \( f \circ g \in \bbL^1\plr{\Gamma} \); the push-forward measure \( \mu \circ g^{-1} \) is defined in  \cref{eq:push_forward}.

\paragraph{Motivation}
By analogy with the Hutchinson operator \( \OpSet \) acting on sets, recall \cref{def:hut-op}, and the Hutchinson-like operator \( \OpMsr \) acting on measures, see \cref{eq:def_op_msr}, one can introduce a so-called Ruelle operator~\cite{fan99} acting on complex-valued functions defined in \( \bbR^n \), for instance in \( \spC\plr{\bbR^n} \), by:
\begin{equation}\label{eq:def-op-fct}
	\OpFct \colon f \in \spC\plr{\bbR^n} \mapsto \sum_{\ell \in \bbL} \mu_\ell \, f \circ S_\ell \in \spC\plr{\bbR^n}.
\end{equation}
The reader will note that, with respect to the operator \( \OpSet \), see \cref{def:hut-op}, compact sets are replaced by functions and the union is replaced by the weighted addition.

\begin{proposition}\label{pro:invariance-integral}
	Given an IFS \( \IFS = \setwt{S_\ell}{\ell \in \bbL} \), its attractor \( \Gamma \), and an invariant measure  \( \mu \)  associated to \( \IFS \), for \( f \in \spC\plr{\Gamma} \), we have
	\[
		\int_\Gamma f \di{\mu}
		= \sum_{\ell \in \bbL} \mu_\ell \int_\Gamma f \circ S_\ell \di{\mu}
		= \int_\Gamma \OpFct\clr{f} \di{\mu},
	\]
	with \( \OpFct \) defined in \cref{eq:def-op-fct}.
\end{proposition}
\begin{proof}
	Given \( f \in \spC\plr{\Gamma} \), we compute, by definition of \( \OpMsr \) in \cref{eq:def_op_msr} and \( \OpMsr\mu=\mu \),
	\begin{align}\label{eq:gamma_f}
		\int_\Gamma f \di{\mu}
		= \int_\Gamma f \di\plr*{\OpMsr\mu}
		= \sum_{\ell \in \bbL} \mu_\ell \int_\Gamma f \di\plr*{\mu \circ S_\ell^{-1}}.
	\end{align}
	From the change of variable formula \cref{eq:change-of-variable} and \( \Supp\plr{\mu} = \Gamma \), we get
	\[
		\int_\Gamma f \di\plr*{\mu \circ S_\ell^{-1}}
		= \int_{S_\ell^{-1}\plr{\Gamma}} f \circ S_\ell \di\mu
		= \int_{S_\ell^{-1}\plr{\Gamma} \cap \Gamma} f \circ S_\ell \di\mu
		= \int_\Gamma f\circ S_{\ell} \di\mu,
	\]
	where the latter identity follows from \cref{lem:kI}. We conclude with \cref{eq:gamma_f}.
\end{proof}

\paragraph{Ruelle operator on spaces of polynomials for affine IFS}
We shall be particularly interested in the action of \( \OpFct \) on polynomials.
From now on, we will limit the discussion to \emph{affine} IFS\@.
We denote by \( \spPolyP \) the space of polynomials of \( n \) variables and, for any \( k \in \bbN \), by \( \spPolyP[k] \) the subspace of polynomials of total degree less or equal to \( k \).
The space of polynomials can be decomposed into a direct sum of homogeneous polynomials of degree \( k \):
for any multi-index \( \vec{\alpha} \in \bbN^n \), we define the monomial \( x^{\vec{\alpha}} \coloneqq x_1^{\alpha_1} \cdots x_n^{\alpha_n} \in \spPolyP \) of total degree \( \abs{\vec{\alpha}}_1 = \alpha_1 + \cdots + \alpha_n \) and
\begin{equation}\label{defPkhom}
	\spPolyHom{k} \coloneqq \vect\setst[\big]{x^{\vec{\alpha}}}{\abs{\vec{\alpha}}_1 =  k},
	\text{ so that }
	\spPolyP[k] = \bigoplus_{0 \leq q \leq k} \, \spPolyHom{q} \text{ and } \spPolyP = \bigoplus_{k \in \bbN} \, \spPolyHom{k}
\end{equation}
Seen as a subspace of \( \spC\plr{\bbR^n} \), \( \spPolyP \) is stable by  \( \OpFct \) and the action of \( \OpFct \) does not increase the total degree, in other words, for all \( p \in \spPolyP[k] \), we have \( \OpFct p \in \spPolyP[k] \).
This means that \( \OpFct \) has an upper triangular structure with respect to the decomposition \cref{defPkhom} of \( \spPolyP \): there exists a
family of linear operators \( \OpFct_{k, q} \in \Hom\plr{{\spPolyHom{k}}, \spPolyHom{q}} \), with \( 0 \leq q \leq k \) such that
\begin{equation}\label{eq:opfct-decomposition}
	\OpFct p = \sum_{0 \leq q \leq k} \OpFct_{k,q} p, \quad \text{for all} \ p \in \spPolyHom{k}.
\end{equation}
In \cref{sec:cubature-S-inv}, we shall need some information about the eigenvalues of \( \OpFct \) restricted to the space of polynomials \( \spPolyP \), which we denote by \( \OpFct_{\spPolyP} \).
Let us denote by \( \sigma\plr{\OpFct_{\spPolyP}} \) the set of eigenvalues of \( \OpFct_{\spPolyP} \) (namely, values \( \lambda\in \bbC \), such that there exists \( p\in \spPolyP \), \( p\neq 0 \), for which \( \OpFct_{\spPolyP} p = \lambda p \)). Due to the upper-triangular like structure of \( \OpFct_{\spPolyP} \), it holds that
\begin{align}\label{eq:sigmaF}
	\sigma\plr{\OpFct_{\spPolyP}} = \bigcup_{k \in \bbN} \sigma\plr[\big]{\OpFct_{k,k}}.
\end{align}
For \( k = 0 \), one easily sees that, as \( \mu_1 + \cdots + \mu_L = 1 \),
\begin{equation}\label{eq:F00}
	\OpFct_{0,0} = \Id,
	\quad (\text{the identity operator in \( \spPolyP[0] \)}),
\end{equation}
which means that \( 1 \) is an eigenvalue of \( \OpFct_{\spPolyP} \).
Furthermore, one can localize the eigenvalues of the operators \( \OpFct_{k,k} \), for \( k\neq 0 \), as made precise in the following lemma.

\begin{lemma}\label{lem:loc-spectrum}
	For any \( k \geq 1 \), \( \sigma\plr{\OpFct_{k,k}} \subset \overline{D}(0, r_k) \) (the closed disk of center \( 0 \) and radius \( r_k \)), where
	\[
		r_k
		\coloneqq \sum_{\ell \in \bbL} \mu_\ell \, \rho_\ell^k
		\leq {{\rho}}_{\max}^k < 1,
		\qquad \rho_{\max} = \max_{\ell\in\bbL}\rho_{\ell}.
	\]
\end{lemma}
\begin{proof}
	With \( S_\ell\plr{x} = A_\ell x + b_\ell \), one sees immediately that for all \( p \in \spPolyHom{k} \), we have \( p \circ S_\ell(x)  = p\plr*{A_\ell x} + R_\ell(x) \) where \( R_\ell \in \spPolyP[k-1] \).
	From this observation, we infer the following identity, valid for all \( p \in \spPolyHom{k} \):
	\[
		\OpFct_{k,k} p(x) = \sum_{\ell \in \bbL} \mu_\ell \, p\plr*{A_\ell x}.
	\]
	Let \( B = \overline{B}(0, 1) \) be the unit closed ball of \( \bbR^n \)  and \( p \in \spPolyHom{k} \) be an eigenvector of \( \OpFct_{k,k} \) associated to \( \lambda \in \bbC \).
	Of course \( \norm{p}_{\spL^\infty(B)} \neq 0 \) and there exists \( x_* \in B \) such that \( \abs{p\plr{x_*}} = \norm{p}_{\spL^\infty(B)} \).
	However, \( \OpFct_{k,k} p = \lambda \, p \) implies \( \OpFct_{k,k} p\plr{x_*} = \lambda \, p\plr{x_*} \), thus
	\[
		\abs*{\lambda \, p\plr*{x_*}}
		\leq \sum_{\ell \in \bbL} \mu_\ell \, \abs*{p\plr*{A_\ell x_*}}
	\]
	and since each matrix \( A_\ell \) is \( \rho_\ell \)-contractive, \( A_\ell x_* \in \rho_\ell B \), thus, by homogeneity,
	\[
		\abs*{p\plr*{A_\ell x_*}}
		\leq \max_{x \in \rho_\ell B} \abs*{p(x)}
		= \max_{y \in B} \abs*{p\plr*{\rho_\ell y}}
		= \rho_\ell^k \, \max_{y \in B} \abs*{p\plr*{y}}
		= \rho_\ell^k \, \abs*{p\plr*{x_*}}.
	\]
	We then conclude the proof since, as \( \abs{\vec{\mu}}_1 = 1 \),
	\[
		\abs{\lambda}
		\leq \sum_{\ell \in \bbL} \mu_\ell \rho_\ell^k
		\leq {{\rho}}_{\max}^k \sum_{\ell \in \bbL} \mu_\ell
		= {{\rho}}_{\max}^k
	\]
\end{proof}

\begin{corollary}\label[corollary]{cor:1-eig-simple} The value \( 1 \) is a simple eigenvalue of the operator \( \OpFct_{\spPolyP} \) (\emph{i.e.}~its algebraic, and thus, geometric, multiplicity equals to 1). Moreover, it is also the largest eigenvalue of \( \OpFct_{\spPolyP} \).
\end{corollary}

\begin{remark}\label[remark]{rem:eigs-exact} In dimension \( n = 1 \), since all the spaces \( \spPolyHom{k} \) are one-dimensional, the operators \( \OpFct_{k,k} \) are reduced to multiplication operators.
	The eigenvalues  \( \sigma\plr{\OpFct_{\spPolyP}} \) of the operator \( \OpFct_{\spPolyP} \) are known explicitly; they are real and positive.
	Indeed, from the proof of \cref{lem:loc-spectrum}, one easily sees that
	\[
		\sigma\plr{\OpFct_{\spPolyP}} = \setwt*{ \mu_1 a_1^k + \cdots + \mu_L a_L^k }{ k \in \bbN }
		\quad \text{where} \
		A_\ell = a_\ell = \pm \rho_\ell.
	\]
	The above reasoning can be extended to a higher dimensional case, when all matrices \( A_\ell \) are diagonal, \( A_\ell = \diag\plr*{a_{\ell,1}, \ldots, a_{\ell,n}} \) with \( \abs*{a_{\ell,i}} \leq \rho_\ell < 1 \).
	In this case, it is readily seen that \( \plr{S_\ell\plr{x}}^{\vec{\alpha}} = a_\ell^{\vec{\alpha}} \, x^{\vec{\alpha}} + R_{\vec{\alpha}} \), where \( a_\ell^{\vec{\alpha}} = a_{\ell,1}^{\alpha_1} \cdots a_{\ell,n}^{\alpha_n} \)
	and \( R_{\vec{\alpha}} \in \spPolyP[k-1] \). Thus, the operator \( \OpFct_{k,k} \) is diagonal in the basis \( \setst{x^{\vec{\alpha}}}{\abs{\vec{\alpha}}_1 = k} \) of \( \spPolyHom{k} \), and
	we have
	\[
		\sigma\plr*{\OpFct_{\spPolyP}} = \bigcup_{k \in \bbN} \setst*{ \mu_1 a_1^{\vec{\alpha}} + \cdots + \mu_L a_L^{\vec{\alpha}} }{ \abs{\vec{\alpha}}_1 = k }.
	\]
\end{remark}

\subsection{Interpolatory cubature on self-affine sets}

\subsubsection{Definition and associated operators}

Let \( K \supseteq \Gamma \) be a non-empty compact set.
While the results of this section apply to any non-empty compact set \( K \supseteq \Gamma \), in particular \( K = \Gamma \), they are of practical interest when \( K \) allows for an easy choice of cubature points, \emph{e.g.}~a hyperrectangle.
As discussed in the introduction, the goal of this article is to construct a cubature for an approximation of the integral of the function \( f \in \spC(K) \) over \( \Gamma \) with the invariant measure \( \mu \).
One of the commonly used types of cubature formulae provides such an approximation of the form
\begin{equation}\label{eq:def-cubature}
	\int_\Gamma f(x) \, \di{\mu}
	\simeq \Quad\clr{f}
	\coloneqq \sum_{1 \leq i \leq M} w_i \, f\plr*{x_i},
\end{equation}
where \( \calX \coloneqq \setwt{x_i}{1 \leq i \leq M} \subset K \) is the set of \emph{cubature points} and the vector \( \vec{w} = \plr{w_1, \ldots, w_M} \in \bbR^M \) is the vector of \emph{cubature weights}.

\begin{remark}
	As the cubature formula depends of on both \( \calX \) and \( \vec{w} \), a better notation would be \( \Quad_{\calX, \vec{w}} \).
	However, we have abandoned this option in order to avoid heavy notation, the role of \( \calX \) and \( \vec{w} \) being implicit.
\end{remark}

Let us remark that common cubature formulae (\emph{e.g.}~designed for \( \Gamma \) being an interval and \( \mu \) being the Lebesgue's measure) rely on a set of points \( \calX \subset \Gamma \).
For self-affine sets, the condition \( \calX\subset \Gamma \) is difficult to ensure, and therefore, we do not provide any constraints on the location of the points \( \calX \), but the ones described below.
This is different from the Gauss quadrature based on the work~\cite{mantica96}, where the location of the points \( \cal X\) depends on the measure \( \mu \).

For constructing \( \Quad \), we shall follow the classic paradigm for constructing a cubature formula (interpolatory cubature), where one requires that ``\( \simeq \)'' becomes ``\( = \)'' in \cref{eq:def-cubature}, when \( f \) belongs to a finite dimensional space \( \calP \subset \spPolyP \) of  polynomials.
The idea behind being that any continuous function can be approximated uniformly by a sequence of polynomials.
In the rest of the article, we assume that \( \spPolyP[0] \subset \calP \).

\begin{remark}[Choice of the space \( \calP \)]\label[remark]{rem_choiseP} The usual practical choice for the space is \( \calP = \spPolyP[k] \) or \( \calP = \spPolyQ{k} \), both choices being particularly well-suited for the error analysis.
	Recall that the space \( \spPolyQ{k} \) is defined as \( \Span\setst{x^{\vec{\alpha}}}{\abs{\vec{\alpha}}_\infty \leq k} \) with \( \abs{\vec{\alpha}}_\infty = \max_{1 \leq i \leq n} \alpha_i \).
	However, all theoretical arguments extend to a general space \( \calP \).
\end{remark}

Given \( \calP \), the usual procedure is first to consider a set \( \calX \coloneqq \setwt{x_i}{1 \leq i \leq M} \), where \( M = \dim\calP \), of  cubature points  which is  \( \calP \)-unisolvent.

\begin{definition}
	The evaluation operator associates to any continuous function the set of its values on \( \mathcal{X} \):
	\[
		\Eval \colon \spC\plr*{\bbR^n} \to \bbR^M, \
		f \mapsto \vec{f} \coloneqq \plr*{f\plr*{x_1}, \ldots, f\plr*{x_M}}.
	\]
	Then, by definition the set of points \( \calX \) is \( \calP \)-unisolvent if, and only if, the map \( \Eval \) is injective from \( \calP \) into \( \bbR^M \).
\end{definition}
If \( \calX \) is  \( \calP \)-unisolvent, \( \Eval_{\calP} \coloneqq \restr{\Eval}{\calP} \) is thus bijective from \( \calP \) to \( \bbR^M \).
The \( \calP \)-unisolvency in \( \cal X \) is a necessary and sufficient condition for constructing the set of Lagrange polynomials \( \Lpoly_i \) in \( \calP \), associated to the set \( \calX \), defined by \( \Lpoly_i\plr*{x_j} = \delta_{i,j} \), which is equivalent to \( \Lpoly_i = \Eval^{-1}_{\calP}\plr*{\vec{e}_i} \),
where \( \blr{\vec{e}_i} \) is the canonical basis of \( \bbR^M \).
The set \( \plr{\Lpoly_i}_{i=1}^M \) is a basis of \( \calP \) and
\begin{equation}\label{eq:decomp-lagrange}
	\forall p \in \calP, \qquad
	p(x) = \sum_{1 \leq i \leq M}  p\plr*{x_i} \; \Lpoly_i(x).
\end{equation}

\begin{remark}
	Even though this is not explicit in the notation, the Lagrange polynomials \( \Lpoly_j \) do depend on \( \calP \) and \( \calX \).
\end{remark}

\begin{remark}\label[remark]{rem:interpolation} The inverse of the operator \( \Eval_\calP \) is nothing but the \emph{interpolation} operator that associates to a vector \( \vec{a} = \plr{a_1, \ldots, a_M} \) the unique polynomial of \( \calP \) that takes the value \( a_i \) at point \( x_i \), namely
	\begin{equation}
		\Interp_{\calP} \colon \bbR^M \to \calP
		\quad \text{ such that } \quad
		\Interp_{\calP}\plr{\vec{a}} = \sum_{1 \leq i \leq M} a_i \, \Lpoly_i.
	\end{equation}
	We shall use later the adjoint operator \( \Interp^*_{\calP} \colon \plr{\calP; \norm{\, \cdot \,}_{\spL^2(\Gamma;\mu)}} \to \bbR^M \) (where \( \bbR^M \) is equipped with the Euclidean inner product).
	The reader will easily verify that
	\begin{align}\label{eq:iadj}
		\Interp^*_{\calP} f = \vec{f},
		\quad \text{where} \ f_i = \int_\Gamma f \, \Lpoly_i \di\mu,
		\quad \text{for} \ i = 1, \ldots, M.
	\end{align}
\end{remark}

Note that the cubature formula \cref{eq:def-cubature} can be rewritten in algebraic form as
\begin{equation}\label{eq:def-quad-2}
	\forall f \in \spC(K),
	\quad
	\Quad\clr{f} = \vec{w} \cdot \Eval(f)
\end{equation}
where \( x \cdot y \) is the inner product in \( \bbR^M \).
Given \( \calP \) and \( \calX \) a \( \calP \)-unisolvent set, we would like to find \( \vec{w} \) in such a way that we integrate exactly all the functions in \( \calP \).
The (trivial) answer is provided by the following lemma.

\begin{lemma}\label{lem:exact-weight}
	Given \( \calP \subset \spPolyP \) and a \( \calP \)-unisolvent set \( \calX \), there exists a unique choice of cubature weights  \( \setwt{w_i}{1 \leq i \leq M} \) such that \( \int_\Gamma p(x) \di{\mu} = \Quad\clr{p} \), for all \( p \in \calP \).
	These cubature weights are given by
	\begin{equation}\label{eq:lagrange-weights}
		w_i = \int_\Gamma \Lpoly_i(x) \di{\mu},
		\qquad \forall 1 \leq i \leq M.
	\end{equation}
\end{lemma}

\Cref{lem:exact-weight} is essentially of theoretical interest since exploiting it would require to compute the integrals in \cref{eq:lagrange-weights}.
This justifies an alternative approach to define cubature weights that we develop in the next sections.

Before proceeding into the definition of the cubature weights, we will need to introduce the notion of \( \IFS \)-invariance of a polynomial space.
We shall see, in \cref{sec:cubature-S-inv}, that cubature weights based on such polynomial spaces can be computed in a purely algebraic manner.
On the other hand, as we shall see in \cref{sec:cubature-non-S-inv}, for polynomial spaces not satisfying this condition, the natural extension of the algebraic method will provide a set of practically useful cubature weights.

\subsubsection{Definition of \texorpdfstring{\( \IFS \)}{S}-invariant spaces}\label{sec:S-inv}

\begin{definition}
	A finite-dimensional subspace \( \calP \) of \( \spPolyP \) is said to be \( \IFS \)-invariant if and only if \( \spPolyP[0] \subset \calP \) and \( \calP \) is stable by the operator \( \OpFct \), meaning that \( \OpFct\plr{\calP} \subset \calP \).
\end{definition}

\begin{remark}
	The reader will easily verify that for any \( \IFS \) and any \( k \), the space \( \spPolyP[k] \) is \( \IFS \)-invariant.
	If, for all \( \ell \), the matrix \( A_\ell \) has in each row and column only one non-zero element, the space \( \spPolyQ{k} \) is \( \IFS \)-invariant.
	However, in general, the space \( \spPolyQ{k} \) is not \( \IFS \)-invariant, as shown in the next example.
\end{remark}

\begin{example}
	Let \( n = 2 \), \( L = 2 \), and the matrices \( A_k \) in \cref{eq:ifs-affine} are defined as follows: \( A_0 = \rho \Id \) and \( A_1 = \rho  \, T \), where
	\[
		T = \tfrac{\sqrt{2}}{2} \begin{pmatrix}
			1 & 1 \\ 1 & -1
		\end{pmatrix}
		\quad \Longleftrightarrow \quad
		T x = \tfrac{\sqrt{2}}{2} \plr*{x_1 + x_2, x_1 - x_2}^\trans.
	\]
	We observe that \( p(x) = x_1^k x_2^k \in \spPolyQ{k} \) but \( p\plr{T x} = \frac{1}{2} \plr{x_1 + x_2}^k \plr{x_1 - x_2}^k \notin \spPolyQ{k} \).
	The space \( \OpFct\plr{\spPolyQ{k}} \) is included into \( \spPolyP[nk] \).
	This inclusion can be strict.
	Indeed, in the above case, for \( k=1 \), we have that, the polynomials \( p_{i,j}(x) = x_1^i x_2^j \), \( i, j \in \blr{0, 1} \), satisfy \( p_{i, j} \circ T \in \spPolyQ{1} \) if \( ij = 0 \), and \( p_{1,1} \circ T = \frac{1}{2}(x_1^2-x_2^2) \).
	The above computation implies that \( \OpFct\plr{\spPolyQ{k}} = \vect\blr{1, x_1, x_2, x_1 x_2, x_1^2 - x_2^2} \), which is a strict subspace of \( \spPolyP[2] \).
\end{example}

\section{Cubature based on \texorpdfstring{\( \IFS \)}{S}-invariant polynomial spaces}\label{sec:cubature-S-inv}

The key assumption of this section, unless stated otherwise, is  that the space \( \calP \) is \( \IFS \)-invariant.

\subsection{Algebraic definition of the cubature weights}

Recall that our goal is to construct a cubature rule \cref{eq:def-cubature}, which would be exact for \( p \in \calP \).
As shown in \cref{lem:exact-weight}, this condition is ensured by the unique choice  \( w_i = \int_\Gamma \Lpoly_i \di{\mu} \), for \( 1 \leq i \leq M \).
On the other hand, by \cref{pro:invariance-integral}, we can relate the integrals of the Lagrange polynomials to the integrals of their images by the Ruelle operator:
\begin{equation}\label{eq:def_weights_key}
	w_i
	= \int_\Gamma \Lpoly_i \di{\mu}
	= \int_\Gamma \OpFct\clr{\Lpoly_i} \di{\mu}.
\end{equation}
As \( \calP \) is \( \IFS \)-invariant, \( \OpFct\clr{\Lpoly_i} \in \calP \), and is thus integrated exactly by the cubature rule \( \Quad \).
In other words, the cubature weights satisfy the following identity:
\begin{equation*}
	w_i
	= \sum_{1 \leq j \leq M} w_j \OpFct\clr{ \Lpoly_i}\plr*{x_j}
	= \sum_{1 \leq j \leq M} \sum_{\ell\in \bbL} \mu_\ell \, \Lpoly_i \circ S_{\ell}\plr*{x_j} \, w_j.
\end{equation*}
In order to write the above in the algebraic form, let us define the  \( M \times M \) matrix
\begin{equation}\label{eq:def-mat-S}
	\vec{S}_{i,j} \coloneqq \sum_{\ell \in \bbL} \mu_\ell \, \Lpoly_j \circ S_\ell\plr*{x_i},
	\qquad
	\text{for} \ 1 \leq i,j \leq M.
\end{equation}
With this definition, the  above said is summarized in the following lemma, which provides \emph{an algebraic property} of the cubature weights.

\begin{lemma}\label{lem:S-inv-weight}
	Assume that \( \calP \) is \( \IFS \)-invariant and that \( \calX \) is \( \calP \)-unisolvent, then the vector \( \vec{w} \) given by \cref{eq:lagrange-weights} is an eigenvector of \( \vec{S}^\trans \),  associated to the eigenvalue \( 1 \):
	\begin{equation}\label{eq:S-inv-weight}
		\vec{S}^\trans \vec{w} = \vec{w}.
	\end{equation}
\end{lemma}

Obviously, \cref{eq:S-inv-weight} is not sufficient to characterize \( \vec{w} \in \Ker\plr{\matI - \vec{S}^\trans} \).
We must complete it by requiring that the cubature formula is exact for the constant function equal to \( 1 \), namely, using the algebraic version \cref{eq:def-quad-2} of \( \Quad \),
\begin{equation}\label{eq:normal}
	\vec{w} \cdot \vec{1} = w_1 + \cdots + w_M = 1,
	\qquad \text{where} \ \vec{1} \coloneqq \Eval(x \mapsto 1).
\end{equation}
Thus, the problem to be solved, to compute the vector \( \vec{w} \) of cubature weights writes
\begin{equation}\label{eq:pb-weights}
	\text{Find} \ \vec{w} \in \Ker\plr{\matI - \vec{S}^\trans} \ \text{satisfying \cref{eq:normal}}.
\end{equation}
By \cref{lem:S-inv-weight}, we know that this problem admits at least one solution \( \vec{w} \) given by \cref{eq:lagrange-weights}.
However, a priori it is unclear that this solution would be unique.
The answer to this question is directly linked to the spectral structure of \( \vec{S}^\trans \) and more precisely to the multiplicity of the eigenvalue \( 1 \).
Because the spectra of the matrices \( \vec{S}^\trans \) and \( \vec{S} \) (together with multiplicities) coincide, we will work with the matrix \( \vec{S} \), which appears to have a convenient structure for the analysis.

\subsection{Properties of the matrix \texorpdfstring{\( \vec{S}^\trans \)}{ST} and well-posedness of \texorpdfstring{\cref{eq:pb-weights}}{\ref{eq:pb-weights}} }

The matrix \cref{eq:def-mat-S} is strongly linked to the Ruelle operator  \( \OpFct \), which we will use to characterize its spectrum.
The first part of the following lemma will be used immediately, while the second part will be of use later.
\begin{lemma}\label{lem:opfct-eval}
	For all \( p \in \calP, \) where the space \( \calP \) is not necessarily \( \IFS \)-invariant, it holds that \( \Eval\plr*{\OpFct p} = \vec{S} \plr*{\Eval p} \).
	We also have \( \Eval_{\calP}\Interp_{\calP}\Eval\plr*{\OpFct p} = \vec{S} \plr*{\Eval p} \).
\end{lemma}
\begin{proof}
	From the definition of \( \OpFct \), see \cref{eq:def-op-fct} and the definition of \( \vec{S} \), see \cref{eq:def-mat-S}, the columns of \( \vec{S} \) are given by \( \vec{S}_{\cdot, j} = \Eval\plr{\OpFct\Lpoly_j} \), for \( 1 \leq j\leq M \).
	Next, with the Lagrange decomposition \cref{eq:decomp-lagrange} of \( p \) and by linearity of \( \Eval \), it holds that
	\begin{align*}
		\Eval\plr{\OpFct p} =
		\Eval \circ \OpFct \! \sum_{1 \leq j \leq M} p\plr*{x_j} \, \Lpoly_j
		= \sum_{1 \leq j \leq M} p\plr*{x_j} \, \Eval\plr*{\OpFct \Lpoly_j}
		= \sum_{1 \leq j \leq M} \vec{S}_{\cdot, j} \, p\plr*{x_j}.
	\end{align*}
	The first statement is obtained using the definition of \( \Eval \).
	The second statement follows by recalling that \( \Eval_{\calP}\Interp_{\calP} = \Id_{\bbR^M} \) and applying \( \Eval_{\calP}\Interp_{\calP} \) to the first identity.
\end{proof}

The \( \IFS \)-invariance of \( \calP \) implies in particular that \( \OpFct_{\calP}=\restr{\OpFct}{\calP} \in \Hom\plr{\calP} \).
Using the fact that \( {\Eval}_{\calP} \) is invertible (recall that its inverse is the interpolation operator \( \Interp_{\calP} \), see \cref{rem:interpolation}), \cref{lem:opfct-eval} implies that
\begin{equation}\label{eq:CD}
	\OpFct_{\calP} =  \Eval_{\calP}^{-1}\vec{S} \Eval_{\calP},
\end{equation}
in other words that the spectra of \( \OpFct_{\calP} \) and \( \vec{S} \) coincide.

\begin{lemma}\label{lem:S-inv-spectrum}
	The eigenvalue \( 1 \) is a simple eigenvalue of \( \vec{S} \) (or, equivalently, \( \OpFct_{\calP} \) according to \cref{eq:CD}) and all other eigenvalues are strictly less than \( 1 \) in modulus.
	Moreover, the corresponding eigenspace is given by \( \Span\blr{\vec{1}} \).
\end{lemma}
\begin{proof}
	Given the \( \IFS \)-invariant and finite-dimensional space \( \calP \), we have \( \calP \subset \spPolyP[k] \) with \( k = \max_{p \in \calP} \deg p \).
	Let \( W \) be a complementing subspace of \( \calP \) in \( \spPolyP[k] \).
	Let \( \OpFct_k \in \Hom\plr{\spPolyP[k]} \) be the restriction of \( \OpFct \) to \( \spPolyP[k] \), the \( \IFS \)-invariance property means that, with respect to the decomposition \( \spPolyP[k] = \calP \oplus W \), the operator \( \OpFct_k \) has a triangular block structure
	\[
		\OpFct_k = \begin{pmatrix}
			\OpFct_{\calP \to \calP} & \OpFct_{W \to \calP}
			\\
			0                        & \OpFct_{W \to W}
		\end{pmatrix}.
	\]
	As a consequence, we have the inclusion \( \sigma\plr*{\OpFct_{\calP \to \calP}} \subset \sigma\plr*{\OpFct_k} \). On the other hand, remark that \( \sigma\plr*{\OpFct_k} = \bigcup_{0 \leq q \leq k}\sigma\plr*{\OpFct_{q,q}} \), \emph{cf.}~\cref{eq:opfct-decomposition}. The announced result about the eigenvalues is easily deduced from \cref{lem:loc-spectrum} and \cref{eq:F00}.

	It remains to show that the corresponding eigenspace is given by \( \Span\blr{\vec{1}} \). This is easily seen from \cref{eq:CD}: for \( \vec{v}\in \bbR^M \), it holds that \( \vec{S}\vec{v} = \vec{v} \) if and only if \( \OpFct_{\calP} \Eval_{\calP}^{-1}\vec{v} = \Eval_{\calP}^{-1}\vec{v} \). On the other hand, \( \OpFct_{\calP}1 = 1 \), and \( \Eval_{\calP}1 = \vec{1} \), and thus \( \vec{v} \) satisfying \( \vec{S}\vec{v} = \vec{v} \) necessarily belongs to \( \Span\blr{\vec{1}} \).
\end{proof}

\begin{remark}
	The proof of \cref{lem:S-inv-spectrum} provides a more precise characterization of the spectrum \( {\OpFct}_{\calP} \), namely,  with \( \OpFct_{k,k} \) defined in the proof of \cref{lem:loc-spectrum}, see \cref{eq:opfct-decomposition}:
	\[
		k \coloneqq \min\setst{k \in \bbN}{\calP \subset \spPolyP[k]}
		\quad \Rightarrow \quad
		\sigma({\OpFct}_{\calP}) \setminus \{1\} \subset \bigcup_{1 \leq q \leq k} \sigma\plr*{\OpFct_{q,q}}
	\]
	where each \( \sigma\plr*{\OpFct_{k,k}} \) is localized in a ``small'' ball centered at the origin (see \cref{lem:loc-spectrum}).
\end{remark}

As a consequence of \cref{lem:S-inv-spectrum}, we have the following result.

\begin{theorem}\label{thm:S-inv-well-posed}
	The problem \cref{eq:pb-weights} is well-posed and characterizes the vector \( \vec{w} \) of cubature weights.
\end{theorem}

The proof of this result relies partially on a standard linear algebra argument, which we repeat for the convenience of the reader and which will be used later in the paper.

\begin{lemma}\label{lem:linear_algebra}
	Assume that \( \vec{A}\in \bbR^{M\times M} \), and let \( \lambda\in \bbR \) be a simple eigenvalue of \( \vec{A} \) (\emph{i.e.}~its algebraic, and thus, geometric multiplicity equals to \( 1 \)).
	Then the corresponding left and right eigenvectors \( \vec{v}_{l} \) and \( \vec{v}_r \) are not orthogonal.
	\emph{I.e.} if \( \vec{v}_l,\, \vec{v}_{r}\in \bbR^M \), \( \vec{v}_l\neq 0 \), \( \vec{v}_r\neq 0 \), are such that \( \vec{v}_l^\trans \vec{A} = \lambda\vec{v}_l^\trans \) and \( \vec{A}\vec{v}_r = \lambda\vec{v}_r \), then, necessarily, \( \vec{v}_l \cdot \vec{v}_r \neq 0 \).
\end{lemma}
\begin{proof}
	We reason by contradiction.
	Assume that \( \vec{v}_l \cdot \vec{v}_r = 0 \).
	Next, recall the decomposition of the space  \( \bbR^M = \Ker\plr{\lambda-\vec{A}}  \oplus_{\perp} \Img\plr{\lambda-\vec{A}^\trans} \).
	Since \( \vec{v}_r \in \Ker(\lambda-\vec{A}) \), and \( \vec{v}_l \cdot \vec{v}_r = 0 \), then, necessarily, \( \vec{v}_l \in \Img(\lambda-\vec{A}^\trans) \).
	Therefore, there exists \( \vec{y} \in \bbR^M \), such that \( (\lambda-\vec{A}^\trans)\vec{y} = \vec{v}_l \).
	Remark that \( \vec{y} \notin \Span\blr{\vec{v}_l} \), and is thus a generalized eigenvector.
	This contradicts the fact that \( \lambda \) is a simple eigenvalue of \( \vec{A}^\trans \) (equivalently, of \( \vec{A} \)).
\end{proof}
\begin{proof}[Proof of \cref{thm:S-inv-well-posed}]
	Because the eigenvalues of \( \vec{S} \) and \( \vec{S}^\trans \) coincide (together with multiplicities), by  \cref{lem:S-inv-spectrum} we conclude that there exists \( \vec{v} \neq 0 \), such that \( \vec{S}^\trans\vec{v} = \vec{v} \), and the corresponding eigenvalue is simple.
	Using \cref{lem:S-inv-spectrum}, we also have that \( \vec{S}\vec{1} = \vec{1} \).
	Applying \cref{lem:linear_algebra} with \( \vec{A} = \vec{S} \), \( \lambda = 1 \) shows that, necessarily, \( \vec{v} \cdot \vec{1} \neq 0 \).

	It is easy to verify that \( \widetilde{\vec{w}} = \plr{\vec{v}\cdot \vec{1}}^{-1}\vec{v} \) is a unique solution to the problem \cref{eq:pb-weights}, which is thus well-posed.
	Since the exact cubature weights satisfy \cref{eq:pb-weights}, necessarily, \( \widetilde{\vec{w}} \) coincides with the vector \( \vec{w} \).
\end{proof}

In what follows, we will need a rewriting of \cref{thm:S-inv-well-posed}, which will serve us in a sequel (see \cref{thm:exact}).

\begin{corollary}\label{cor:reinterp_l2}
	Consider the following problem: Find \( \omega\in \calP \) such that
	\begin{align}\label{eq:pb_l2}
		\int_\Gamma \omega \, p \di\mu
		= \int_\Gamma \omega \OpFct p \di\mu,
		\ \forall p \in \calP,
		\quad \text{and} \quad
		\int_\Gamma \omega \di\mu = 1.
	\end{align}
	This problem is well-posed, and the unique solution is given by \( \omega = 1 \).
\end{corollary}
\begin{proof}
	First remark that \( \omega = 1 \) is a solution thanks to \cref{pro:invariance-integral} and \( \mu\plr{\Gamma}  = 1 \).
	Using \( \plr{p, q} \mapsto \int_\Gamma p \, \bar{q} \di\mu \) as an inner product on \( \calP \), we have the orthogonal decomposition \( \calP = \Span\blr{1} \oplus W \), where \( W = \setst{p \in \calP}{\int_\Gamma p \di\mu = 0} \).
	For the uniqueness, let \( \omega \) be another solution to \cref{eq:pb_l2} and denote \( \omega' = \omega - 1 \).
	We have \( \int_\Gamma \omega' \, \plr{\Id - \OpFct} p \di\mu = 0 \), for all \( p \in \calP \), \emph{i.e.}~\( \omega' \perp (\Id-\OpFct)W \), and \( \int_\Gamma \omega' \di\mu = 0 \), \emph{i.e.}~\( \omega' \in W \).
	From \cref{pro:invariance-integral}, we have \( \int_\Gamma p \di\mu = \int_\Gamma \OpFct p \di\mu \) for all \( p \in \calP \), thus \( \OpFct\plr{W} \subset W \).
	The operator \( \restr{\Id - \OpFct}{W} \) is injective by \cref{cor:1-eig-simple}.
	Since it maps \( W \) into itself, \( \plr{\Id - \OpFct}W = W \), thus \( \omega' \perp W \).
	Therefore, we have \( \omega' = 0 \).
\end{proof}

\begin{remark}
	Remarkably, given two different sets of cubature points \( \calX \) and \( \mathcal{Y} \) each of them being \( \calP \)-unisolvent, the spectra of the corresponding matrices \( \vec{S}_\calX \) and \( \vec{S}_\calY \) (which differ a priori) are the same.
	Indeed, these matrices are representations of the same operator \( \OpFct_{\calP} \), \emph{cf.}~\cref{eq:CD}.
\end{remark}

\section{Cubature based on non-\texorpdfstring{\( \IFS \)}{S}-invariant polynomial spaces}\label{sec:cubature-non-S-inv}
One motivation for looking at the case of  non-\( \IFS \)-invariant polynomial spaces stems from a practical difficulty of choosing cubature points that would be unisolvent in \( \spPolyP[k] \), and possess Lebesgue's constants of mild growth, with respect to the number of points (this property is favorable for convergence of cubatures formulae, as we will see in \cref{sec:error_polynomial}).
Therefore, instead, we will look at the tensor product spaces \( \spPolyQ{k} \), which has an additional favorable property: for well-chosen unisolvent sets of points, the associated Lagrange polynomials are easy to compute using barycentric formulae, \emph{cf.}~\cite{atap}.
It appears that the spaces \( \spPolyQ{k} \) in general are not \( \IFS \)-invariant, as discussed in \cref{sec:S-inv}. Therefore, a priori, the cubature weights on the space \( \spPolyQ{k} \) no longer satisfy the algebraic property \cref{eq:pb-weights}.
This is illustrated numerically in \cref{sec:numeric}.

\subsection{An alternative definition of the cubature weights}

Let \( \calP \subset \spPolyP \) with \( \dim\calP < +\infty \) and \( \spPolyP[0] \subset \calP \).
According to \cref{lem:exact-weight}, the cubature weights \( \vec{w} \) given by \cref{eq:lagrange-weights} provide a cubature formula \( \Quad\clr{f} \) which is exact in \( \calP \).
In the previous section, under the assumption that \( \calP \) is \( \IFS \)-invariant, we have shown that the corresponding weights can be characterized as a unique solution to the purely algebraic problem \cref{eq:pb-weights}.
However, this is no longer true if the spaces \( \calP \) are not \( \IFS \)-invariant.

Our idea is to abandon the constraint imposed by the cubature formula of being exact in \( \calP \), but instead use \cref{eq:pb-weights} as a characterization of the cubature weights, in the same way as the property \cref{eq:pb-weights} characterized \textit{exact} cubature weights for \( \IFS \)-invariant spaces, see \cref{thm:S-inv-well-posed}.
More precisely, instead of the exact weights \( \vec{w} \), defined in \cref{lem:exact-weight}, we will look for  weights \( {\widetilde{\vec{w}}} \) as a solution to the following problem:
\begin{equation}\label{eq:pb-weights-approximate}
	\text{Find} \ \widetilde{\vec{w}} \in \Ker\plr{\matI - \vec{S}^\trans} \ \text{satisfying} \ \widetilde{\vec{w}}\cdot \vec{1} = 1.
\end{equation}
Recall that the constraint \( \widetilde{\vec{w}}\cdot \vec{1}=1 \) in  \cref{eq:pb-weights-approximate} ensures that \( \spPolyP[0] \) is integrated exactly. This new cubature rule, defined by
\begin{align*}
	\int_{\Gamma}f \di{\mu}
	\approx \widetilde{\Quad}[f]
	\coloneqq \sum_{1 \leq i \leq M} \widetilde{w}_i f\plr*{x_i},
\end{align*}
is consistent with the invariance property  \cref{pro:invariance-integral} of the integral, in sense of
\begin{proposition}\label{prop:qf}
	If \( \widetilde{\Quad}\clr{f} = \sum_{i=1}^M \widetilde{w}_i f\plr{x_i} \),  the following two properties are equivalent: (i) for all \( p \in \calP \), \( \widetilde{\Quad}\clr{p} = \widetilde{\Quad}\clr{\OpFct p} \) and (ii) \( \widetilde{\vec{w}}\in \Ker\plr{\matI - \vec{S}^\trans} \).
\end{proposition}
We leave the proof of this result to the reader.
Following the above discussion, we single out two important questions about the ``new'' cubature weights \cref{eq:pb-weights-approximate}:
\begin{itemize}[leftmargin=*]
	\item the well-posedness of the problem \cref{eq:pb-weights-approximate};
	\item the error committed by replacing the cubature rule \( \Quad \) with its perturbed version \( \widetilde{\Quad} \) when integrating polynomials \( p \in \calP \).
\end{itemize}

\subsection{On the well-posedness of the problem \texorpdfstring{\cref{eq:pb-weights-approximate}}{\ref{eq:pb-weights-approximate}}}

Unfortunately, we were able to show only a partial well-posedness result.
\begin{theorem}\label{lem:main_standard}
	The value \( \lambda = 1 \) is an eigenvalue of the matrix \( \vec{S}^\trans \).
	If this eigenvalue is simple, then the associated eigenvector \( \widetilde{\vec{w}} \) satisfies \( \widetilde{\vec{w}} \cdot \vec{1} \neq 0 \).
\end{theorem}
\begin{proof}
	The first statement follows by remarking that \( \vec{S}\vec{1} = \vec{1} \) (the result follows by the same reasoning as in the proof of the second statement of \cref{lem:S-inv-spectrum}). To prove the second statement, let us assume that \( \widetilde{\vec{w}} \neq 0 \) is such that \( \vec{S}^\trans \widetilde{\vec{w}} = \widetilde{\vec{w}} \). Applying  \cref{lem:linear_algebra} with \( \vec{A} = \vec{S} \) and \( \lambda = 1 \) implies that \( \widetilde{\vec{w}} \cdot \vec{1} \neq 0 \).
\end{proof}
According to the above result's proof, the condition that \( \lambda = 1 \) is a \emph{simple} eigenvalue ensures that \cref{eq:pb-weights-approximate} is well-posed, in particular, the constraint \( \widetilde{\vec{w}} \cdot \vec{1} = 1 \) can be satisfied. One can verify that since \( \vec{S} \) is real, \( \widetilde{\vec{w}} \) satisfying \cref{eq:pb-weights-approximate} is necessarily real, too.

\subsection{Subspaces of \texorpdfstring{\( \calP \)}{P} integrated exactly by the inexact cubature}

Since the error of the cubature is closely related to the largest polynomial space integrated exactly by the cubature, it is natural to ask a question of characterization of spaces of polynomials in \( \calP \) that are integrated exactly, \emph{i.e.}
\begin{equation}\label{defPQ}
	\calP_{\Quad} \coloneqq \setst*{ p \in \calP }{ \int_\Gamma p \di{\mu} = \widetilde{\Quad}\clr{p} } .
\end{equation}
Since the integral and \( \widetilde{\Quad} \) are linear forms, this space has co-dimension 1 (thus dimension \( \Dim\calP-1 \)) in \( \calP \) independently of the vector \( \widetilde{\vec{w}} \) (!).
A more precise characterization of the space \( \calP_{\Quad} \) is given in the theorem below.

\begin{theorem}\label{thm:exact}
	Assume that \( \vec{S}^\trans \widetilde{\vec{w}} = \widetilde{\vec{w}} \) and \( \widetilde{\vec{w}} \cdot \vec{1} = 1 \). Let \( \calP_\IFS \)  be the largest \( \IFS \)-invariant subspace of \( \calP \).
	Then the inexact cubature formula \( \widetilde{\Quad} \)  integrates polynomials in the space \( \calP_\IFS \) exactly, in other words, \( \calP_\IFS \subset \calP_{\Quad} \).
\end{theorem}
\begin{proof}
	First, let us argue that we can rewrite the algebraic problem \cref{eq:pb-weights-approximate} in the variational form resembling the statement of \cref{cor:reinterp_l2}.
	Recall \cref{eq:pb-weights-approximate} satisfied by \( \widetilde{\vec{w}}\in \bbR^M \), rewritten in an equivalent form, see \cref{lem:opfct-eval}:
	\begin{align}\label{eq:sec_eq}
		\widetilde{\vec{w}}^\trans
		= \widetilde{\vec{w}}^\trans\vec{S}
		\iff \widetilde{\vec{w}} \cdot \Eval {p}
		= \widetilde{\vec{w}} \cdot \Eval_\calP \Interp_\calP \Eval \OpFct p,
		\qquad \forall p \in \calP.
	\end{align}
	Remark that in the above identity, it is the operator \( \Interp_\calP\Eval\OpFct \) that appears, rather than \( \OpFct \).
	Moreover, because of non-\( \IFS \)-invariance of \( \calP \), \( \Interp_\calP\Eval\OpFct\neq \OpFct \) on \( \calP \).
	By definition of \( \Interp_\calP^* \), we have the following correspondence between \( L^2\plr{\Gamma} \) and \( \bbR^n \) Euclidean inner-product
	\begin{align*}
		\text{for all } f, \, p \in \calP, \qquad
		\int_{\Gamma} f \, p \, \di\mu = \sum_{1\leq i\leq M} p\plr*{x_i} \int_\Gamma f \, \Lpoly_i \di{\mu}
		= \Interp_{\calP}^*f \cdot \Eval p.
	\end{align*}
	Thus, denoting \( \widetilde{\vec{w}} = \Interp_\calP^* \widetilde{\omega} \), \emph{i.e.}~\( \widetilde{\omega} \coloneqq \plr{\Interp_\calP^*}^{-1} \widetilde{\vec{w}} \), the second equality of \cref{eq:sec_eq} rewrites
	\begin{equation}\label{eq:omega_def}
		\int_{\Gamma} \widetilde{\omega} \, p \, \di\mu = \int_{\Gamma}\widetilde{\omega}\Interp_\calP\Eval\OpFct p \di{\mu}, \quad \forall p \in \calP.
	\end{equation}
	Next, let us decompose the space \( \calP \) into two subspaces orthogonal with respect to \( L^2(\Gamma) \)-inner-product: \( \calP=\calP_{\IFS}\oplus_{\perp}\calP_{\IFS}^{\perp} \).
	In particular, any \( v\in \calP \) can be written as orthogonal sum \( v=v_{\IFS}+v_{\perp} \), where \( v_{\IFS}\in \calP_{\IFS} \) and \( \int_{\Gamma} v_{\perp} \, v_{\IFS} \di{\mu} = 0 \). Similarly, we rewrite \( \widetilde{\omega}=\widetilde{\omega}_{\IFS}+\widetilde{\omega}_{\perp} \).
	Applying this decomposition to \cref{eq:omega_def}, and restricting the space of test functions to \( \calP_{\IFS} \) yields the new identity for \( \widetilde{\omega}_{\IFS} \):
	\begin{align}\label{eq:omegaIFS}
		\int_\Gamma \widetilde{\omega}_{\IFS}p \di{\mu}
		= \int_\Gamma \widetilde{\omega}_{\IFS} \OpFct p \di{\mu},
		\qquad \forall p\in \calP_{\IFS},
	\end{align}
	where we used the fact that \( \Interp_\calP \Eval\OpFct p = \OpFct p \) for all \( p\in \calP_{\IFS} \): since \( \calP_{\IFS} \) is \( \IFS \)-invariant, \( \Eval\OpFct p \in \calP \) when \( p \in {\cal P}_\IFS \).
	Moreover, from \( \widetilde{\vec{w}} \cdot \vec{1} = 1 \), and \( \spPolyP[0] \subset \calP_{\IFS} \),  we have
	\begin{align}\label{eq:omegaav}
		\int_\Gamma \widetilde{\omega}_{\IFS} \di{\mu} = 1.
	\end{align}
	In \cref{eq:omegaIFS}, \cref{eq:omegaav}, we recognize the problem of \cref{cor:reinterp_l2}, where \( \calP \) is replaced by \( \calP_{\IFS} \).
	Therefore, \( \widetilde{\omega}_{\IFS} = 1 \). Since for all \( p\in\cal{P}_{\IFS} \), it holds that
	\[
		\widetilde{\Quad}[p]
		= \widetilde{\vec{w}} \cdot \Eval p = \int_\Gamma \widetilde{\omega} p \di{\mu}
		= \int_\Gamma \widetilde{\omega}_{\IFS} p \di{\mu}
		= \int_\Gamma p \di{\mu},
	\]
	where the second equality above follows from the definition \( \widetilde{\omega} = \plr{\Interp_{\calP}^*}^{-1} \widetilde{\vec{w}} \), and the third is a consequence of the orthogonality, we conclude that \( \calP_{\IFS} \subset \calP_{\Quad} \).
\end{proof}
The above result may seem somewhat surprising: indeed, even if the matrix \( \vec{S}^\trans \) admits several eigenvectors satisfying the assumptions of \cref{thm:exact}, the cubature rule based on any such eigenvector is exact in \( \calP_{\IFS} \).

\begin{remark}
	The above result enables us to use the inexact cubature weights to compute the weights exact in the space \( \calP_{\IFS} \).
	For example, to integrate polynomials in the space \( \spPolyQ{k} \) exactly, one could have computed the inexact weights in \( \spPolyQ{nk} \).
	Such weights would have integrated the polynomials in the space \( \spPolyP[nk] \) (which is \( \IFS \)-invariant) exactly, in particular the Lagrange polynomials for \( \spPolyQ{k} \), which define the exact cubature weights.
	However, this can be potentially computationally expensive.
\end{remark}

\section{Error estimates}\label{sec:error}

In this section, we will quantify the approximation error of \( \int_\Gamma f \di{\mu} \) using the cubature rule, \( \Quad \) or \( \widetilde{\Quad} \), introduced in the previous sections.
Such estimates will be obtained from quite standard arguments for cubature's error estimates and cubature on regular sets.
We suggest two versions of approximation of \( \int_\Gamma f \di{\mu} \):
\begin{itemize}[leftmargin=*]
	\item The \( h \)-version, where, very roughly speaking, we decompose \( \Gamma \) into smaller subsets \( \mu\plr{\Gamma_j} \leq h \), and approximate \( \int_{\Gamma_j} f \di{\mu} \) by a cubature of a fixed degree \( M \).
	      Convergence will be assured by taking \( h \to 0 \).

	\item The \( p \)-version, where we approximate directly \( \int_\Gamma f \di{\mu} \) by the cubature of the degree \( M \).
	      Convergence is then assured by taking \( M \to +\infty \).
\end{itemize}
In both cases, the error analysis relies on the following (classic) estimate, which we recall here.
Let \( \calP_{\Quad} \) be defined as in \cref{defPQ}.
Then
\begin{align}\label{eq:error_cubature}
	\abs*{e[f]} \leq \inf_{p \in \calP_{\Quad}} \norm*{f-p}_{L^{\infty}(K)} \, \clr*{1 + \abs*{\vec{w}}_1},
	\quad \text{where} \quad
	e[f] \coloneqq \int_\Gamma f \di{\mu} - \Quad[f]
\end{align}
which follows from the fact that, for any \( p \in \calP_{\Quad} \), \( e[f] = \int_\Gamma (f-p) \di{\mu} - \Quad[f-p] \).

\subsection{Convergence by refining the mesh}

\subsubsection{Nested meshes on the IFS}\label{sec:h_version}

To define a mesh on \( \Gamma \), it is natural to decompose it into sub-fractals using the IFS contractive maps. Let us fix the mesh size \( h>0 \). The first, simplest decomposition, can be obtained from the following observation.
Since \( \Gamma = \OpSet\plr{\Gamma} \) is a fixed point of the Hutchinson operator \( \OpSet \), see \cref{thm:attractor-existence}, it is also a fixed point of the \( p \)-times iterated  operator \( \OpSet^p \).
This gives the following decomposition of \( \Gamma \), using \( \vec{m} = \plr{m_1, \ldots, m_p} \in \bbL^p \):
\begin{equation}\label{eq:gamma_decomp}
	\Gamma = \bigcup_{\vec{m} \in \bbL^p} \Gamma_{\vec{m}}
	\qquad \text{where}\
	\Gamma_{\vec{m}} = S_{m_1} \circ \cdots \circ S_{m_p}\plr{\Gamma}.
\end{equation}
The maximal refinement level \( p \) is chosen so that \( \Diam \Gamma_{\vec{m}}<h \),  for all  \( \vec{m}\in \bbL^p \).
Remark that in the above, a priori \( \mu(\Gamma_{\vec{m}}\cap \Gamma_{\vec{m'}})\neq 0 \) for \( \vec{m}\neq \vec{m}' \in \bbL^p \).
The use of the mesh \cref{eq:gamma_decomp} for the integral evaluation can be computationally inefficient, since the diameters of the patches \( \Gamma_{\vec{m}} \) in the above decomposition can vary significantly (in other words, some regions of \( \Gamma \) are refined too finely).
Therefore, instead we will use an approach suggested in~\cite{gibbs2023}.
Let us introduce some useful notation (with \( \vec{m} \in \bbL^p \))
\begin{align*}
	 & S_{\vec{m}} \coloneqq S_{m_1} \circ \cdots \circ S_{m_p}, \quad
	\rho_{\vec{m}} \coloneqq \rho_{m_1} \cdots \rho_{m_p}, \quad
	\mu_{\vec{m}} \coloneqq \mu_{m_1} \cdots \mu_{m_p},
	\\
	 & \mathbf{L}^p = \bbL^1 \cup \cdots \cup  \bbL^p, \quad \text{for any } p \in \bbN^* \cup \blr{\infty}.
\end{align*}
For \( p = 0 \), by convention, we define \( \bbL^0 = \blr{\emptyset} \), \( \Gamma_\emptyset = \Gamma \), \( S_\emptyset = \Id \), \( \rho_\emptyset = 1 \), and \( \mu_\emptyset = 1 \).
Given \( h > 0 \), let us define the set of multi-indices \( \calL_h \subset  \mathbf{L}^\infty \) by the following:
\begin{align}\label{eq:mdef}
	\vec{m} \in \calL_h \iff \rho_{\vec{m}} \Diam\Gamma \leq h, \quad \text{and} \quad
	\rho_{\ell}^{-1}\rho_{\vec{m}}\Diam\Gamma > h, \ \text{for some} \ \ell \in \bbL,
\end{align}
and the corresponding decomposition of \( \Gamma \) via
\begin{align}\label{eq:new_decomp}
	\Gamma = \bigcup_{\vec{m}\in\mathcal{L}_h} \Gamma_{\vec{m}}.
\end{align}

We will see below that these definitions and decompositions make sense.
Let us remark that in \cref{eq:mdef}, we used \( \rho_{\vec{m}} \Diam\Gamma \) instead of \( \Diam\Gamma_{\vec{m}} \), because for self-affine sets \( \Diam\Gamma_{\vec{m}} \) is not easily computable.
On the other hand, \( \Diam\Gamma_{\vec{m}}\leq \rho_{\vec{m}} \Diam\Gamma \).
The first condition in \cref{eq:mdef} thus ensures that \( \Diam\Gamma_{\vec{m}} < h \) for all patches \( \Gamma_{\vec{m}} \), while the second condition is a requirement of minimality of the set \( \calL_h \), in the sense that the patches \( \Gamma_{\vec{m}} \) cannot become too small.

The set \( \calL_h \) can be obtained by the following algorithm that constructs a sequence \( \calL_h^q \), with \( \calL_h^q \subset{\mathbf{L}}^q \), which is stationary after a finite number of steps. Then, \( \calL_h \) is defined as the stationary point of this algorithm.
The algorithm proceeds as follows.
Starting from \( \calL_h^0 = \blr{\emptyset} \), we deduce \( \calL_h^{q+1} \) from \( \calL_h^q \) by looking at the (possibly empty) set \( \calN_h^q = \setst{\vec{m} \in \calL_h^q}{\rho_{\vec{m}} \Diam\Gamma > h} \) and setting
\begin{align}\label{eq:Lk}
	\calL_h^{q+1} = \plr*{\calL_h^q \setminus \calN_h^q} \cup \setwt*{\plr*{m_1, \ldots, m_q, \ell}}{\plr*{m_1, \ldots, m_q} \in \calN_h^q \ \text{and} \ \ell \in \bbL}.
\end{align}
The reader will easily  verify that \emph{(1)} the stationary point
\( \calL_h \) is characterized by \cref{eq:mdef}; \emph{(2)} for all \( q \geq 0 \),  \( \Gamma = \bigcup_{\vec{m} \in \mathcal{L}_h^q}\Gamma_{\vec{m}} \) (by induction on \( k \)), which implies \cref{eq:new_decomp}.

\begin{remark}\label[remark]{remark:stop} Setting \( \rho_{\max} = \max_{\ell\in\bbL}\rho_{\ell} \), the algorithm stops after \( q_{\max} \) iterations with \( q_{\max} \leq q_* \coloneqq \min\setst{q \in \bbN}{\rho_{\max}^q \Diam\Gamma \leq h} \).
\end{remark}

\begin{remark}
	For similar IFS, it holds that \( \rho_{\min}^p h < \Diam\Gamma_{\vec{m}} \leq h \), for all \( \vec{m} \in \calL_h \), where \( \rho_{\min} = \min_{\ell\in\bbL}\rho_{\ell} \) and \( p \) the length of \( \vec{m} \).
\end{remark}

Just like in \cref{eq:gamma_decomp}, \( \Gamma_{\vec{m}} \) and \( \Gamma_{\vec{m}'} \) in \cref{eq:new_decomp} may have non-trivial intersections.
However, this decomposition defines a ``partition'' of \( \Gamma \) in the sense of the following result.

\begin{lemma}\label{lem:int-decom-h}
	For \( h > 0 \), \( f \in \spL^1\plr{\Gamma;\mu} \), and \( q \geq 0 \), we have
	\begin{equation}\label{eq:int-leaves}
		\int_\Gamma f\plr{x} \di{\mu}
		= \sum_{\vec{m} \in \calL_h^q} \mu_{\vec{m}} \int_\Gamma f \circ S_{\vec{m}} \di{\mu}
		\quad \text{and} \quad
		\sum_{\vec{m} \in \calL_h^q} \mu_{\vec{m}} = 1.
	\end{equation}
\end{lemma}
\begin{proof}
	We proceed by induction.
	The result is true for \( q = 0 \) (trivial); for \( q = 1 \), we remark that \( \calL_h^1 = \bbL \). Then \cref{eq:int-leaves} is direct from \cref{pro:invariance-integral}. Next, assuming \cref{eq:int-leaves}, by definition \cref{eq:Lk} of \( \calL_h^{q+1} \), it holds that
	\begin{align*}
		\sum_{\vec{m} \in \calL_h^{q+1}} \mu_{\vec{m}} \int_\Gamma f \circ S_{\vec{m}} \di{\mu}
		 & = \sum_{\vec{m} \in \calL_h^q \setminus \calN_h^q} \mu_{\vec{m}} \int_\Gamma f \circ S_{\vec{m}} \di{\mu}
		\\
		 & + \sum_{\vec{m} \in \calN_h^q} \mu_{\vec{m}} \sum_{\ell \in \bbL} \mu_\ell \int_\Gamma f \circ S_{\vec{m}} \circ S_\ell \di{\mu}.
	\end{align*}
	By \cref{pro:invariance-integral}, \( \sum_{\ell \in \bbL} \mu_\ell \int_\Gamma f \circ S_{\vec{m}} \circ S_\ell \di{\mu} = \int_\Gamma f \circ S_{\vec{m}} \di{\mu} \), thus
	\begin{align*}
		\sum_{\vec{m} \in \calL_h^{q+1}} \mu_{\vec{m}} \int_\Gamma f \circ S_{\vec{m}} \di{\mu}
		 & = \sum_{\vec{m} \in \calL_h^q \setminus \calN_h^q} \mu_{\vec{m}} \int_\Gamma f \circ S_{\vec{m}} \di{\mu}
		+ \sum_{\vec{m} \in \calN_h^q} \mu_{\vec{m}} \int_\Gamma f \circ S_{\vec{m}} \di{\mu},
		\\
		 & = \sum_{\vec{m} \in \calL_h^q} \mu_{\vec{m}} \int_\Gamma f \circ S_{\vec{m}} \di{\mu},
	\end{align*}
	which proves the first part of \cref{eq:int-leaves}.
	The second part follows from taking \( f = 1 \).
\end{proof}

\subsubsection{Mesh refinement cubature}

Let \( K \supset \Gamma \) be a compact of \( \bbR^n \), and \( \calX \coloneqq \blr{x_i}_{i=1}^M \subset K \) be a set of points that are \( \calP \)-unisolvent on \( K \).
Let us assume that the weights \( \blr{w_i}_{i=1}^M \) are such that the cubature rule \( \Quad \) is exact on \( \spPolyP[k] \subset \calP \) for some \( k \in \bbN \).
Following \cref{lem:int-decom-h}, we define a cubature rule \( \Quad_h \):
\begin{align}\label{eq:new_cub}
	\Quad_h\clr{f} \coloneqq \sum_{\vec{m} \in \calL_h} \mu_{\vec{m}} \Quad\clr*{f \circ S_{\vec{m}}}.
\end{align}
We define the sets
\begin{equation}\label{defKh}
	K_h = \bigcup_{\vec{m} \in \calL_h} S_{\vec{m}}\plr{K}
	\quad \text{and} \quad
	\quad F_h = \bigcup_{\delta \leq h} K_\delta.
\end{equation}
It can be shown that \(F_h\) is a decreasing sequence of compact sets and that  \( F_h \supset K_h \) converges to \( \Gamma \) as \( h \to 0 \), in the Hausdorff distance.
By construction, the \( h \)-cubature formula \cref{eq:new_cub} uses the values of the function \( f \) in the  set \( K_h \), rather than the whole compact \( K \) (the interest of this is discussed in \cref{rem:Kh}).
Accordingly, we have the following result, which is an extension to high order cubature of the corresponding result~\cite[Thm.~3.6]{gibbs2023}.

\begin{theorem}\label{thm:h-version}
	Let \( K \) be a convex compact set and satisfy \( \Gamma, \calX \subset K \).
	Let \( H > 0 \), \( k \in \bbN \), and let a cubature \( \Quad \) be exact on \( \spPolyP[k] \).
	For all \( f \in \spC^{k+1}\plr{F_H} \) and \( h \leq H \), we have
	\[
		\abs*{\Quad_h\clr*{f} - \int_\Gamma f \di{\mu}}
		\leq  C \, h^{k+1} \norm*{\D^{k+1} f}_{\spL^\infty\plr*{F_H}},
	\]
	where
	\( \norm{\D^{k+1} f}_{\spL^\infty\plr{F_H}} = \max_{\abs{\vec{\beta}}_1 = k+1} \norm{\partial^{\vec{\beta}} f}_{\spL^\infty\plr{F_H}} \)
	and the constant \( C \) is defined by
	\[
		C = \plr*{\frac{\Diam K}{\Diam \Gamma}}^{k+1}\plr*{\abs*{\vec{w}}_1 + 1} \sum_{\abs{\vec{\beta}} = k+1} \frac{1}{\vec{\beta}!}.
	\]
\end{theorem}

\begin{remark}\label[remark]{rem:Kh} As seen from the statement of the above result, the convergence of the cubature depends on the regularity of the integrand in the set \( F_H \).
	This allows to use efficiently such an \( h \)-cubature to integrate functions that have singularities inside a convex set \( K \), but regular in a small vicinity of \( \Gamma \).
\end{remark}

\begin{proof}[Proof of \cref{thm:h-version}]
	Let \( f \in \spC^{k+1}\plr{F_H} \).
	Using \cref{lem:int-decom-h} and \cref{eq:new_cub},
	\begin{equation}\label{eq:main_bound}
		\abs*{\Quad_h\clr{f} - \int_\Gamma f \di{\mu}}
		= \abs*{\sum_{\vec{m}\in \calL_h}\mu_{\vec{m}} \clr*{\Quad\clr*{f\circ S_{\vec{m}}} -\int_\Gamma f\circ S_{\vec{m}} \di{\mu}}}
		\leq \max_{\vec{m} \in \calL_h} e_{\vec{m}}
	\end{equation}
	where
	\[
		e_{\vec{m}}
		=\abs*{\Quad\clr*{f \circ S_{\vec{m}}} - \int_\Gamma f \circ S_{\vec{m}} \di{\mu}}.
	\]
	To estimate each of the terms \( e_{\vec{m}} \), we use \cref{eq:error_cubature}, where we recall that the polynomials \( \spPolyP[k] \) are integrated exactly:
	\begin{align}\label{eq:polexp}
		\abs*{e_{\vec{m}}} \leq \inf_{p \in \spPolyP[k]} \norm*{f\circ S_{\vec{m}}-p}_{L^{\infty}(K)} \clr*{1 + \abs*{\vec{w}}_1}.
	\end{align}
	To estimate the right-hand side, we use the multivariate Taylor series expansion of \( f \) in the point \( z_{\vec{m}}={S}_{\vec{m}}z \), with arbitrary \( z\in K \). It holds that  \( f = t_{\vec{m}} + r_{\vec{m}} \), where
	\begin{align}
		t_{\vec{m}}(y)
		 & = \sum_{\abs{\vec{\alpha}}_1 \leq k} \frac{\partial^{\vec{\alpha}}f\plr*{z_{\vec{m}}}}{\vec{\alpha}!} \plr*{y - z_{\vec{m}}}^{\vec{\alpha}},
		\qquad
		r_{\vec{m}}(y)
		= \sum_{\abs{\vec{\beta}}_1 = k+1} h_{\vec{\beta}}\plr{y} \plr*{y - z_{\vec{m}}}^{\vec{\beta}},
		\label{eq:tmrm}
		\\
		h_{\vec{\beta}}(y)
		 & = \frac{k+1}{\vec{\beta}!} \int_0^1 \plr{1-t}^k D^{\vec{\beta}} f(z_{\vec{m}}+t(y-z_{\vec{m}})) \di{t}.
		\label{eq:hbeta}
	\end{align}
	Remark that \( t_{\vec{m}}\in \spPolyP[k] \), and thus \( t_{\vec{m}}\circ S_{\vec{m}}\in \spPolyP[k] \); then \cref{eq:polexp} implies that
	\begin{align}\label{eq:polexp2}
		|e_{\vec{m}}|\leq  \norm*{r_{\vec{m}}\circ S_{\vec{m}}}_{L^{\infty}(K)} \clr*{1 + \abs*{\vec{w}}_1}.
	\end{align}
	From \cref{eq:tmrm} it follows that
	\begin{equation}\label{eq:bound_tmp1}
		\norm{r_{\vec{m}} \circ S_{\vec{m}}}_{\spL^{\infty}(K)}
		\leq \sum_{\abs{\vec{\beta}}_1 = k+1} \norm*{h_{\vec{\beta}}}_{\spL^{\infty}(S_{\vec{m}}(K))} \abs*{S_{\vec{m}}y - S_{\vec{m}}z}^{k+1}_2,
	\end{equation}
	where we used the bound \( \abs{\plr{y-z_{\vec{m}}}^{\vec{\beta}}} \leq \abs{y-z_{\vec{m}}}_2^{\abs{\vec{\beta}}_1} \).
	Using \cref{eq:hbeta}, convexity of \( K \) and the fact that \( z_{\vec{m}}=S_{\vec{m}}z\in K_h \), we conclude that
	\begin{equation}\label{eq:bound_tmp2}
		\norm*{h_{\vec{\beta}}}_{\spL^\infty\plr*{S_{\vec{m}}\plr{K}}}
		\leq \frac{1}{\vec{\beta}!} \norm*{\partial^{\vec{\beta}} f}_{\spL^\infty\plr*{K_h}}
		\leq \frac{1}{\vec{\beta}!} \norm*{\D^{k+1} f}_{\spL^\infty\plr*{F_H}}.
	\end{equation}
	On the other hand,
	\begin{align}\label{eq:bound_tmp3}
		\abs*{S_{\vec{m}}y-S_{\vec{m}}z}^{k+1}_2
		\leq \rho_{\vec{m}}^{k+1} \abs{y-z}^{k+1}
		\leq \plr*{\frac{h \Diam K}{\Diam \Gamma}}^{k+1},
	\end{align}
	where we use the fact that \( \vec{m}\in \calL_h \), and thus \( \rho^{\vec{m}}\Diam \Gamma\leq h \).
	Combining \cref{eq:bound_tmp2,eq:bound_tmp3} into \cref{eq:bound_tmp1}, inserting the resulting bound into \cref{eq:polexp2} and using \cref{eq:main_bound} yields the desired result.
\end{proof}

\subsection{Convergence by increasing the cubature order}\label{sec:error_polynomial}
Alternatively, we can evaluate the cubature by computing \( \Quad[f] \) directly, without splitting \( \Gamma \) into sub-sets.
In this case the convergence is achieved by increasing the number of cubature points \( M \) (we will denote the corresponding cubature by \( \Quad^{M}\clr{f} \)).
The goal of this section is to quantify the speed of convergence with respect to \( M \).
Recall the estimate \cref{eq:error_cubature}, where we emphasize the dependence of all quantities on \( M \), by using appropriate indices:
\begin{align}\label{eq:error_cubature_recall}
	\abs*{e^M[f]}
	\leq \inf_{p \in \calP_{\Quad}^{M}} \norm*{f-p}_{\infty} \clr*{1 + \abs*{\vec{w}^M}_1}.
\end{align}
Unlike in the previous section, both factors  in the above estimate will play an important role, since, in general, \( \abs{\vec{w}^M}_1 \) can be unbounded as \( M\to +\infty \).
We will start our discussion with an estimate on \( \abs{\vec{w}^M}_1 \).
Next, we present convergence estimates for \( \spC^k \)-functions in \cref{thm:approx_error} and finally estimates for analytic functions in \cref{thm:approx_error2}.

Many of the estimates will be made more explicit for the tensor-product cubature rule satisfying the following assumption.

\begin{assumption}\label[assumption]{asm:quad_tensor} Assume that \( K = \Pi_{\vec{a}, \vec{b}} \supset \Gamma \), where \( \Pi_{\vec{a}, \vec{b}} =\icc{a_1}{b_1} \times \cdots \times \icc{a_n}{b_n} \).
	Let additionally \( \calX =  \calX_1^N  \times \cdots \times \calX_n^N \), where \( \calX_i^N \) are \( (N+1) \) Chebyshev (first or second kind) quadrature nodes on \( \icc{a_i}{b_i} \)~\cite[Chapter 2 and Exercise 2.4]{atap}.
	The corresponding unisolvent polynomial space is the tensor product space \( {\cal P} = \spPolyQ{N} \) with \( M = \plr{N+1}^n \).
	We denote by \( \vec{w}_{C}^{N,n} \) and  \( \Quad_C^{N,n} \) (the index \( C \) stands for ``Chebyshev'') the respective cubature weights and cubature rule, computed with  the method  of this paper.
\end{assumption}

\begin{remark}
	Remark that the cubature points depend on the choice of the compact \( K \), \emph{i.e.} \( K = \Pi_{\vec{a}, \vec{b}} \).
	We choose the Chebyshev points in particular because, for this case, we are able to obtain the most optimal error estimates (this is related to the behavior of the bounds on \( \abs{\vec{w}^M}_1 \), which we connect in \cref{sec:vecw} to the Lebesgue constants of the points \( \calX_i^N \), and that behave asymptotically optimally for Chebyshev points).
\end{remark}

\subsubsection{An estimate on \texorpdfstring{\( \abs{\vec{w}^M}_1 \)}{w1} for \texorpdfstring{\( \IFS \)}{S}-invariant spaces}\label{sec:vecw}

In this case, the weights provided with our algebraic approach satisfy \( w_i^M = \int_\Gamma \Lpoly_i \di{\mu} \) (\emph{cf.}~\cref{sec:cubature-S-inv}) and \cref{thm:S-inv-well-posed}.
These weights may not be positive (see the experiments in \cref{sec:numeric_weights}), and therefore, even though they sum to \( 1 \), \( \abs{\vec{w}^M}_1 \) can still be unbounded as \( M \to +\infty \).
Nonetheless, this quantity grows very mildly, provided a suitable choice of the cubature points is made.
This is summarized in the following result, whose proof is left to the reader.

\begin{lemma}\label{lem:exact}
	The exact weights \( w_j^M = \int_\Gamma \Lpoly_j \di{\mu} \) satisfy the following bound:
	\begin{align*}
		\abs{\vec{w}^M}_1
		= \sum_{1 \leq j \leq M} \abs{w_j^M}
		\leq \Lambda_{\calX}^1(\Gamma) \coloneqq \sup_{x\in \Gamma} \sum_{1 \leq i \leq M} \abs*{\Lpoly_i(x)},
	\end{align*}
	where \( \Lambda_{\calX}^1(\Gamma) \) is the Lebesgue's constant for \( \calX \).
\end{lemma}
In particular, for cubature based on \cref{asm:quad_tensor}, we have the following bound.
\begin{proposition}\label{prop:weight_tensor}
	For the cubature as in \cref{asm:quad_tensor}, there exists \( C(n)>0 \), \emph{s.t.}
	\begin{equation*}
		\big|\vec{w}_C^{N,n}\big|_1 \leq C(n) \log^n (N+1), \quad \text{ for all }N\geq 1.
	\end{equation*}
\end{proposition}
\begin{proof}
	We use the bound of \cref{lem:exact}.
	In our setting Lagrange polynomials are tensor products of univariate Lagrange polynomials, thus
	\begin{align*}
		\Lambda_{\calX}^\Gamma
		\leq \Lambda_{\calX}^K
		\leq \prod_{1 \leq i \leq n} \Lambda_{\calX_i}^{\icc*{a_i}{b_i}},
	\end{align*}
	where \( \Lambda_{\calX_i}^{\icc{a_i}{b_i}} \) is a Lebesgue constant for interpolation in \( \calX_i \) on \( \icc{a_i}{b_i} \), \emph{cf.}~\cite[Ch.~15]{atap}.
	It remains to apply to the above the exact asymptotic of Lebesgue's constant of the Chebyshev points, \emph{cf.}~\cite{brutman97} and references therein.
\end{proof}

\subsubsection{Error analysis for \texorpdfstring{\( \spC^k \)}{Ck}-functions}

The goal of this section is to provide error estimates  for the cubature rule \( Q^M[f] \), provided that \( f\in \spC^k({K}) \). Recall the error bound \cref{eq:error_cubature_recall}.  To obtain convergence estimates, it remains to quantify
\begin{align}\label{eq:error_cubature_recall_0}
	\abs*{e^M_{\calP}[f]} \coloneqq \inf_{p \in \calP_{\Quad}^M} \norm{f-p}_{\spL^\infty(K)},
\end{align}
where \( \calP_{\Quad}^M=\calP_{\Quad} \) is defined in \cref{defPQ}.
The following estimate is fairly well-known.

\begin{theorem}\label{thm:approx_error}
	Let \( K \) be a convex compact set in \( \bbR^n \).
	If \( f \in \spC^k({K}) \), and if \( \spPolyP[N] \subset \calP_{\Quad}^M \), for some \( N\geq 1 \), then there exists a constant \( C(K,n,k) \) independent of \( f \) and \( N \), such that
	\begin{equation}\label{eq:EP0}
		\abs*{e^M_{\calP}[f]}
		\leq \frac{C(K,n,k)}{N^k} \sum_{\abs{\vec{\alpha}}_1 = k} \norm*{\partial^{\vec{\alpha}} f}_{\spL^\infty(K)}.
	\end{equation}
\end{theorem}
\begin{proof}
	Please see the multivariate version of Jackson's theorem as stated in~\cite[Thm.~4.10]{schultz} (see p. 168 of~\cite{schultz} for the justification that \( \operatorname{int} K \) is a regular set, as required by~\cite[Thm.~4.10]{schultz}), \emph{cf.}~as well~\cite[Thm.~2]{bagby_bos_levenberg} for a similar result.
	In the first reference, the result is formulated for \( \spPolyQ{N} \) instead of \( \spPolyP[N] \); nonetheless, it is still valid for \( \spPolyP[N] \) (with a different constant compared to~\cite[Thm.~4.10]{schultz}), since \( \spPolyQ{\lfloor N / n\rfloor}\subset \spPolyP[N] \).
\end{proof}

Combining the estimates of \cref{thm:approx_error} and \cref{lem:exact} yields the following result.

\begin{theorem}\label{thm:approx_error2}
	Suppose that \cref{asm:quad_tensor} holds true.
	Assume that \( f \in \spC^k({K}) \).
	Then, there exists \( C(f, K, n, k) > 0 \), such that for all \( N \geq 1 \),
	\begin{equation}\label{estimate-invariant}
		\abs*{ Q_C^{N,n}[f]-\int_{\Gamma}f \di{\mu} }
		\leq C(f, K, n, k) \plr*{1 + \abs*{\vec{w}_C^{N,n}}_1} N^{-k}.
	\end{equation}
	In the case when \( \calP^{N,n}_C \) are \( \IFS \)-invariant for all \( N \geq 0 \), the above estimate yields
	\begin{equation}\label{estimate-noninvariant}
		\abs*{ Q_C^{N,n}[f]-\int_{\Gamma}f \di{\mu} }
		\leq c \cdot C(f, K, n, k) \log^n (N+1) N^{-k}, \quad \text{ with some }c>0.
	\end{equation}
\end{theorem}
\begin{proof}
	Both results follow from the expression \cref{eq:error_cubature_recall} combined with \cref{thm:approx_error}.
	The estimate in the \( \IFS \)-invariant case is a corollary of \cref{prop:weight_tensor}.
\end{proof}

Compared to the best approximation error estimates obtained for the tensor-product polynomials approximation of smooth functions in the domains \( \Pi_{i=1}^n \icc{a_i}{b_i} \), see~\cite{schultz}, we see that in the \( \IFS \)-invariant case our estimates are worse by a factor of \( \log^n (N+1) \), due to a potential growth of cubature weights.

\subsubsection{Error analysis for analytic functions}

In this section, we fully restrict our attention to cubatures satisfying \cref{asm:quad_tensor}.
We will also work with analytic functions.
This is of interest when evaluating regular integrals in boundary element methods, \emph{e.g.}~\( \int_{\Gamma_{\vec{m}}} \int_{\Gamma_{\vec{n}}} G(x,y) \di{\mu}_x \di{\mu}_y \), where \( G \) is the fundamental solution of the Helmholtz equation.
The ideas and definitions that we present here are now standard, and we follow the exposition in~\cite[Sec.~5.3.2.2]{sauter_schwab}, which we simplify and adapt to our setting.
We will work with componentwise analytic functions, as defined below.

\begin{definition}\label{def:pw_analytic}
	The function \( f \colon \Pi_{\vec{a}, \vec{b}}\to \bbC \) is componentwise analytic, if there exists \( \zeta > 1 \), such that, for each \( i \), and for all \( \vec{y}_i \coloneqq (y_1, \ldots, y_{i-1}, y_{i+1}, \ldots, y_n) \) with \( y_\ell \in [a_\ell,b_\ell] \), the function
	\[
		f_{\vec{y}_i}(t) \coloneqq f\plr*{y_1, \ldots, y_{i-1}, t, y_{i+1}, \ldots, y_n} \colon \icc{a_i}{b_i} \to \bbC,
	\]
	admits an analytic extension \( F_{\vec{y}_i}(z) \) in the  Bernstein ellipse
	\begin{align*}
		\scrE_{a_i,b_i}^{\zeta} = \setwt*{\frac{a_i+b_i}{2}+\frac{a_i-b_i}{2}\frac{z+z^{-1}}{2}}{z \in B_{\zeta}(0)},
	\end{align*}
	and is additionally continuous in the closure of this ellipse.\end{definition}

Next, to estimate \cref{eq:error_cubature}, in particular, the expression \cref{eq:error_cubature_recall_0}, we will choose a particular polynomial from \( \calP^M_{\Quad} \).
To construct it, let us define a Chebyshev weighted \( \spL^2 \)-projection operator of degree \( m \) in the direction \( i \), which is helpful  because we work with tensorized cubature rules.
In particular, let \( \blr{T_k(t)}_{k=0}^\infty \), \( T_k(t) = \cos(k \arccos t) \), \( t \in \icc{-1}{1} \), be the set of Chebyshev polynomials of the first kind, see~\cite[Ch.~3 and Thm.~3.1]{atap} for the relevant discussion.
For each given \( \vec{y}_i \) as in \cref{def:pw_analytic}, we set
\begin{align*}
	\Pi_i^{(m)}      & \colon \plr{\spC(K), \norm{\, \cdot \,}_\infty} \to \plr{\spC(K), \norm{\, \cdot \,}_\infty},
	\quad
	\Pi_i^{(m)} f\plr*{\vec{y}} \coloneqq \sum_{k=0}^{m}a_{\vec{y}_i, k} T_k(y_i),
	\\
	a_{\vec{y}_i, k} & \coloneqq c_k \int_{-1}^1 T_k(t) f_{\vec{y}_i}\plr*{\tfrac{b_i-a_i}{2}t+\tfrac{a_i+b_i}{2}} \plr{1-t^2}^{-1/2} \di t,
\end{align*}
where \( c_0 = 2 / \pi \) and \( c_k = 1 / \pi \) for \( k > 0 \).
Let us introduce the product operator
\begin{align}\label{eq:prod_approx}
	\vec{\Pi}^{(m)}f \coloneqq \Pi_{1}^{(m)} \Pi_2^{(m)} \cdots \Pi_n^{(m)}f \in \spPolyQ{m}.
\end{align}
We shall use this operator to prove the following approximation result.

\begin{proposition}\label{prop:error_f}
	Assume that \( f \) is as in \cref{def:pw_analytic}.
	Assume that \( \spPolyQ{N}\subset \calP^M_{\Quad} \), \( N\geq 1 \).
	Then there exists \( C_f>0 \), such that
	\begin{align*}
		e_{\calP}^M=\inf_{p\in \calP_{\Quad}^M}\norm*{f-p}_{L^{\infty}(K)} \leq C_f \, \log^{n-1}(N+1)\,  \zeta^{-N} / (\zeta-1).
	\end{align*}
\end{proposition}
\begin{proof}
	We start by bounding \( e_\calP^M \leq e_N \coloneqq \norm{f-\vec{\Pi}^{(N)}f}_{L^{\infty}(K)} \).
	Since \( f-\vec{\Pi}^{(N)}f = f-\Pi^{(N)}_1f + \Pi^{(N)}_1f - \Pi_{1}^{(N)}\Pi_2^{(N)}\cdots \Pi_n^{(N)}f \), by the triangle inequality,
	\begin{equation*}
		e_N
		\leq \norm*{f - \Pi^{(N)}_1 f}_{L^{\infty}(K)}
		+ \norm*{\Pi_1^{(N)}} \norm*{f-\Pi_2^{(N)}\cdots \Pi_n^{(N)}f}_{L^{\infty}(K)}.
	\end{equation*}
	Repeating this procedure \( n-2 \) times, with different operators \( \Pi_k^{(N)} \), yields the bound:
	\[
		e_N
		\leq \norm*{f-\Pi_1^{(N)}f}_{L^{\infty}(K)}
		+ \sum_{k=1}^{n-1} \plr*{ \prod_{\ell=1}^k \norm*{\Pi_{\ell}^{(N)}} }  \norm*{f-\Pi_{k+1}^{(N)}f}_{L^{\infty}(K)}.
	\]
	One concludes with the help of the following bounds
	\[
		\norm*{ \Pi_i^{(N)} } \leq C \log (N+1)
		\quad \text{and } \quad
		\norm*{f-\Pi_{k+1}^{(N)}f}_{L^{\infty}(K)} \leq C_f \, \zeta^{-N} / (\zeta-1),
	\]
	easily deduced from existing 1D bounds in~\cite{powell_67,mason_80}
	and~\cite[Thm.~8.2]{atap} respectively.
\end{proof}

Combining \cref{prop:error_f} and \cref{lem:exact} allows quantifying the cubature error.

\begin{theorem}\label{thm:error_analytic}
	Assume that \( f \) is like in \cref{def:pw_analytic} with \( \zeta > 1 \).
	In the situation of \cref{asm:quad_tensor}, there exists a constant \( C_f>0 \), such that the following holds true.
	If the spaces \( \spPolyQ{N} \), \( N\geq 1 \), are \( \IFS \)-invariant, then,
	\begin{equation}\label{esti_pversion_Sinvariant}
		\abs*{\Quad_C^{N,n}[f] - \int_\Gamma f \di{\mu}}
		\leq \frac{C_f \; \zeta^{-N}}{\zeta-1} \log^{2n-1} (N+1).
	\end{equation}
	Otherwise,
	\begin{equation}\label{esti_pversion_Snoninvariant}
		\abs*{\Quad_C^{N,n}[f] - \int_\Gamma f \di{\mu}}
		\leq \frac{C_f \zeta^{-\lfloor N/n\rfloor}}{\zeta-1} \plr*{ 1 + \abs*{\vec{w}^{N,n}_C}_{1} }
	\end{equation}
\end{theorem}
\begin{proof}
	Assume that \( \spPolyQ{N} \), \( N \geq 1 \), are \( \IFS \)-invariant.
	Then \cref{esti_pversion_Sinvariant} follows from \cref{eq:error_cubature_recall}, \cref{prop:weight_tensor} and  \cref{prop:error_f}. If \( \spPolyQ{N} \) is not \( \IFS \)-invariant, by \cref{thm:exact}, the respective cubature rule integrates the polynomials in \( \spPolyP[N] \) exactly.
	Because \( \spPolyQ{\lfloor N/n \rfloor}\subset \spPolyP[N] \), we can again combine \cref{eq:error_cubature_recall} and \cref{prop:error_f} to get \cref{esti_pversion_Snoninvariant}.
\end{proof}

\begin{remark}\label[remark]{rem_optimalite} In the \( \IFS \)-invariant case, compared to the multivariate integration on product domains based on the tensor-product Gauss quadrature (\emph{cf.}~\cite[Thm.~5.3.13]{sauter_schwab}), our cubature error is slightly worse, by a factor \( \log^{2n-1}(N+1) \).
	Moreover, according to the estimate \cref{esti_pversion_Snoninvariant}, the non-\( \IFS \)-invariant case might suffer from a deteriorated convergence rate. We think that this estimate is non-optimal, due to the technique of proof based on \( \spPolyQ{N} \) spaces.
	We believe that one could recover, under stronger assumptions on \( f \) than the ones in \cref{def:pw_analytic}, the same convergence rate as for the \( \IFS \)-invariant case by adapting the approach of~\cite{trefethen_17} to \( \spPolyP[N] \) spaces.
	This approach uses a more elaborate theory of analytic functions of several variables, which is out of scope of this paper. See also the numerical results of \cref{sec:h_conv}.
\end{remark}

\section{Numerical experiments}\label{sec:numeric}

In this section we present several numerical experiments, illustrating statements of different results in our paper, as well as performance of the new cubature rule.
In practice, the compact set \( K \supset \Gamma \) chosen is a hyperrectangle and the points \( \calX \) are chosen on a Cartesian grid, see \cref{asm:quad_tensor}.

\subsection{Algorithmic realization and implementation}
We implemented the cubature rule following \cref{asm:quad_tensor}, as a \texttt{Julia} code, see~\cite{ZM_code}.
In all experiments we used Chebyshev points of the first kind.

\paragraph{Finding \( K \)}
Finding a bounding box for an IFS is a subject of research~\cite{CHU2003407,Mar2009}.
The bounding box is the smallest hyperrectangle \( K = \Pi_{i=1}^n \icc{a_i}{b_i} \supset \Gamma \).
It is either explicitly known (for ``classic'' fractals like Cantor dust), or can be computed as the smallest hyperrectangle containing the fractal.
To approximate the latter, we used a chaos game approach where we start from the fixed points \( \blr{c_\ell}_{\ell \in \bbL} \) and generate new points by applying the map \( S_\ell \) using the ``probabilities'' \( \mu_\ell \).
We obtained the approximate bounding box by considering the smallest hyperrectangle containing all the generated points after large number of iterations.

\paragraph{Computation of cubature weights}
To compute the cubature weights based on the algebraic characterization \cref{eq:pb-weights}, we need to evaluate the entries of the matrix \( \vec{S} \) (Lagrange polynomials), and solve the corresponding eigenvalue problem.
The Lagrange polynomials were evaluated using the barycentric interpolation formula from~\cite[Ch.~5]{atap}.
Recall solving the eigenvalue problem in the \( \IFS \)-invariant case amounts to finding the eigenvector of \( \vec{S}^{\trans} \) corresponding to the largest (in modulus) eigenvalue \( \lambda = 1 \), \emph{cf.}~\cref{lem:S-inv-spectrum}.
Thus, we used the power iteration method, which converges geometrically fast, with a \( 10^{-14} \) absolute residual.
We used the same method in the non-\( \IFS \)-invariant case.
Despite the absence of theory, in all our numerical experiments, \( \lambda=1 \) appeared to be the largest in modulus and a simple eigenvalue.

\begin{remark}
	To compute the weights, we need to assemble the matrix \( S^\trans \).
	This requires \( \OO\plr{L M^3} \) operations.
	The iterative power method converges with rate \( \rho_{\max} = \max_{\ell \in \bbL} \rho_\ell \), see \cref{lem:S-inv-spectrum,lem:loc-spectrum} in the \( \IFS \)-invariant case. Based on the numerical experiments, we conjecture that this is true in the non-\( \IFS \)-invariant case as well.
	In order to have an error on the residual of order \( \delta \coloneqq \norm{\vec{S}^\trans \vec{w} - \vec{w}}_2 \ll 1 \), we need to do \( \OO\plr{\log\delta / \log\rho_{\max}} \) iterations.
	As each iteration costs \( \OO\plr{M^2} \), the total cost of computing the weights is
	\[
		\OO\plr*{L M^3 + M^2 \log\delta / \log\rho_{\max}}.
	\]
\end{remark}

\begin{remark}
	To estimate the cost of computing an integral using our cubature rule, we start by assuming that the absolute values of weights grow ``slowly'', more precisely, \( \abs{\vec{w}_C^{N,n}}_1 = \OO\plr{N^\eta} \), as \( N \to +\infty \), for some \( \eta > 0 \). This holds in particular in the \( \IFS \)-invariant case, when  Chebyshev points quadrature is used to construct the tensorized cubature, \emph{cf.}~\cref{prop:weight_tensor}. We can estimate a cost of a cubature via the number of function evaluations, denoted \( E \), with respect to the desired error \( \abs{\Quad\clr{f} - \int_\Gamma f \di{\mu}} \leq \eps \), as \( \eps \to 0 \).
	For the  \( p \)-version, \( \Quad = \Quad_C^{N,n} \), we have \( E = M \), where \(M\) is the number of cubature points.
	For the \( h \)-version, \( \Quad = \Quad_h \), using \cref{remark:stop}, we obtain \( E = \OO\plr{h^{d_*}} \), where \( d_* = -\log L / \log\rho_{\max} \).
	Therefore, we get for a \( \spC^k\plr{K} \) function:
	\[
		E = \OO\plr*{\eps^{-\frac{n}{k-\eta}}} \ (\text{\( p \)-ver.~\cref{thm:approx_error2}})
		\quad \text{and} \quad
		E = \OO\plr*{\eps^{-\frac{d_*}{k}}} \ (\text{\( h \)-ver.~\cref{thm:h-version}}).
	\]
	For an analytic function on \( K \):
	\[
		E = \OO\plr*{\abs{\log\eps}^n} \ (\text{\( p \)-ver.~\cref{thm:error_analytic}}).
	\]
\end{remark}

\paragraph{Reference solution and errors}
Reference values for the integrals were always computed using a highly-refined \( h \)-version of the method of order \( \OO\plr{h^{15}} \), as described in \cref{sec:h_version}.
The depicted errors are relative errors.

\paragraph{Self-similar sets and measures}
Many of our calculations were performed on different Vicsek-type fractals, defined on \( \bbR^2 \) using \( 5 \) contracting maps:
\[
	S_0 x = R_{\theta}x, \qquad
	S_{\ell} x = \rho x + (1-\rho) c_{\ell}, \quad
	\ell = 1, \ldots, 4, \quad
	c_\ell = \plr*{\pm 1, \pm 1},
\]
where \( \rho = 1 / 3 \), \( R_{\theta} \) is a rotation by angle \( \theta \), see also~\cref{fig:vicsek}.
We use the self-similar measure with \( \mu_\ell = \frac{1}{5} \),  \( \ell \in \bbL \).
\begin{figure}[hbtp]
	\centering
	\subfloat[\( \theta = 0 \)]{\label{fig:vicsek-0}
		\includegraphics[width=0.25\textwidth]{{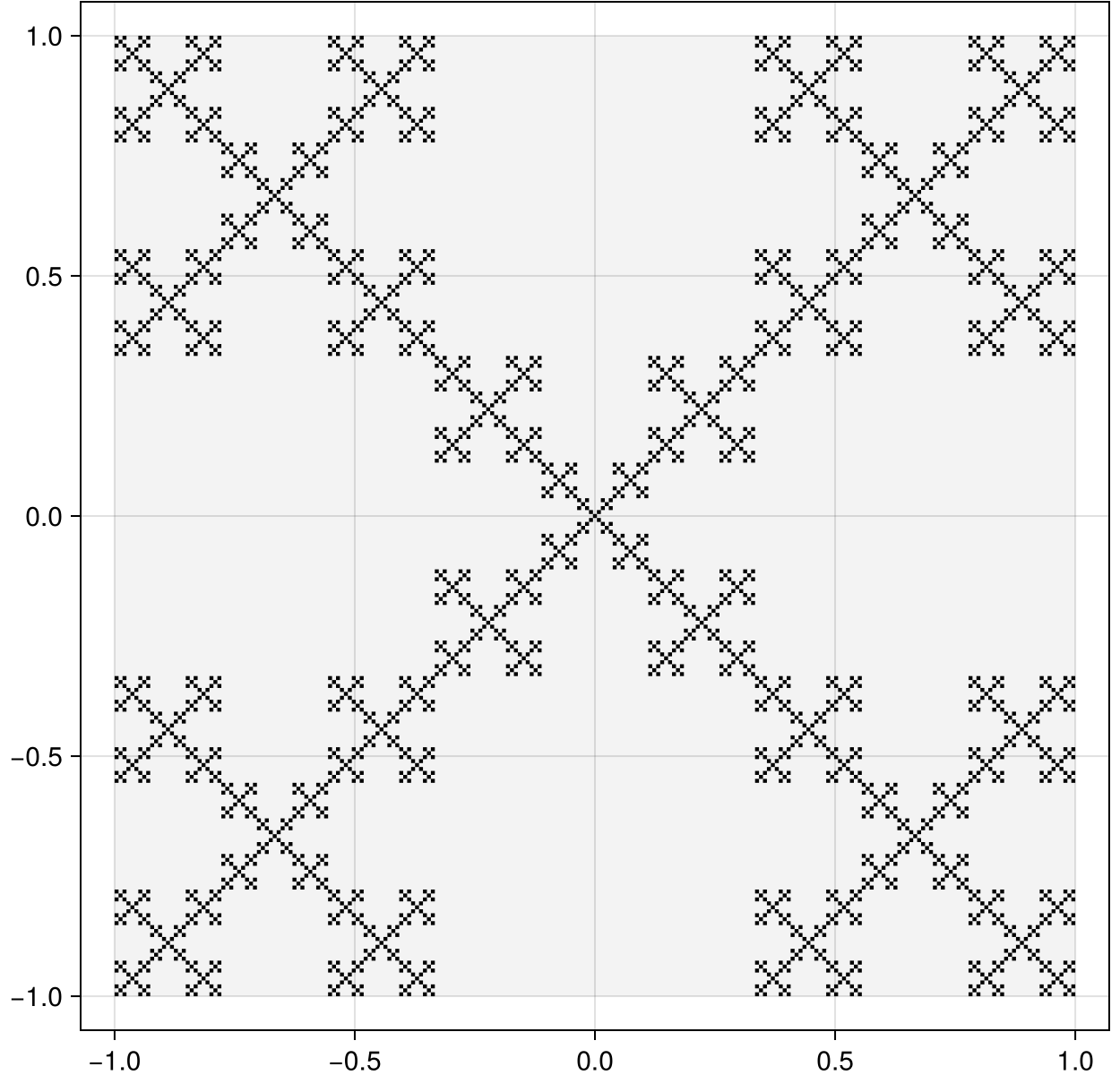}}
	}
	\subfloat[\( \theta = 0.4 \)]{\label{fig:vicsek-0.4}
		\includegraphics[width=0.25\textwidth]{{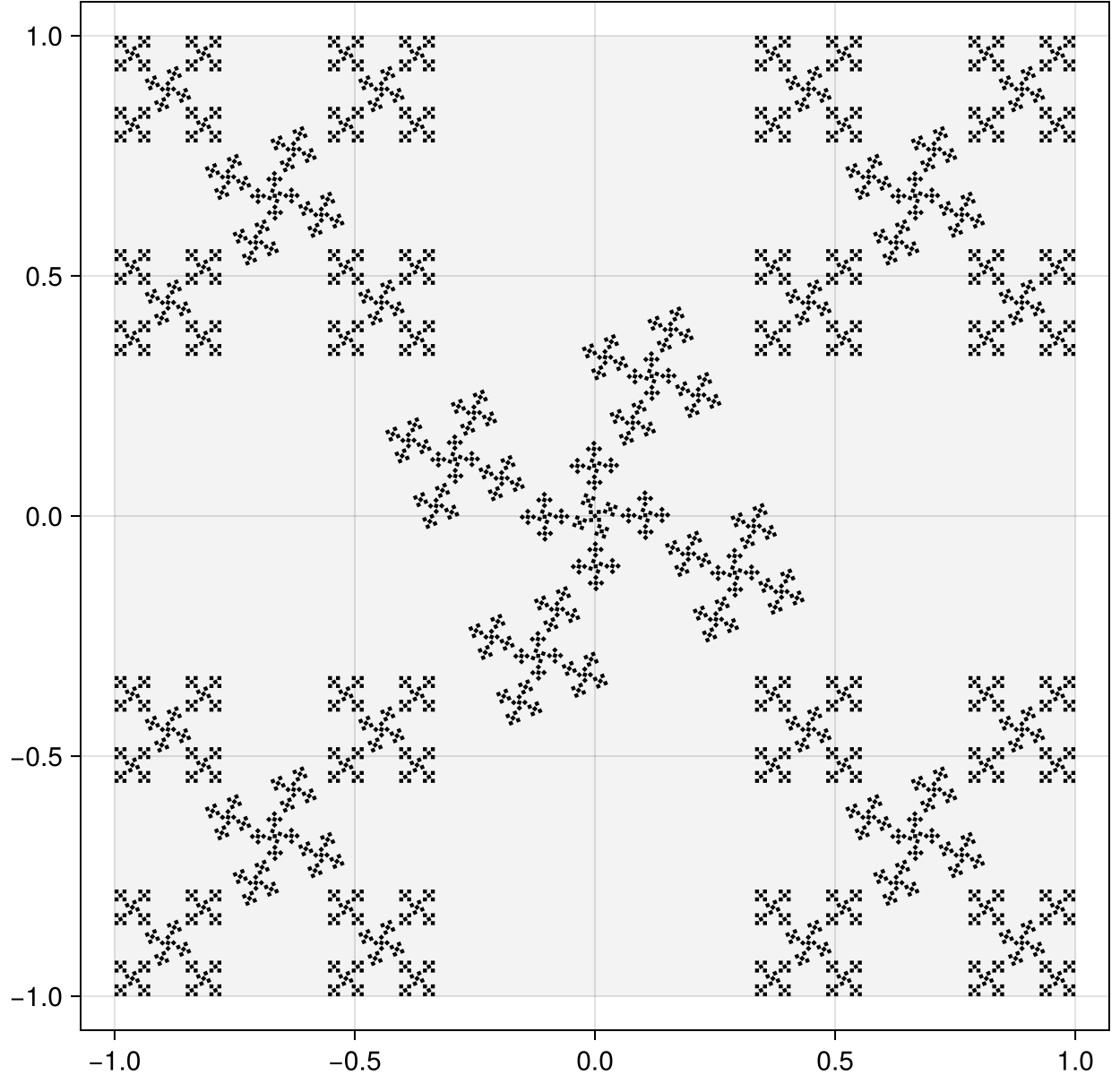}}
	}
	\subfloat[\( \theta = \pi / 4 \)]{\label{fig:vicsek-pio4}
		\includegraphics[width=0.25\textwidth]{{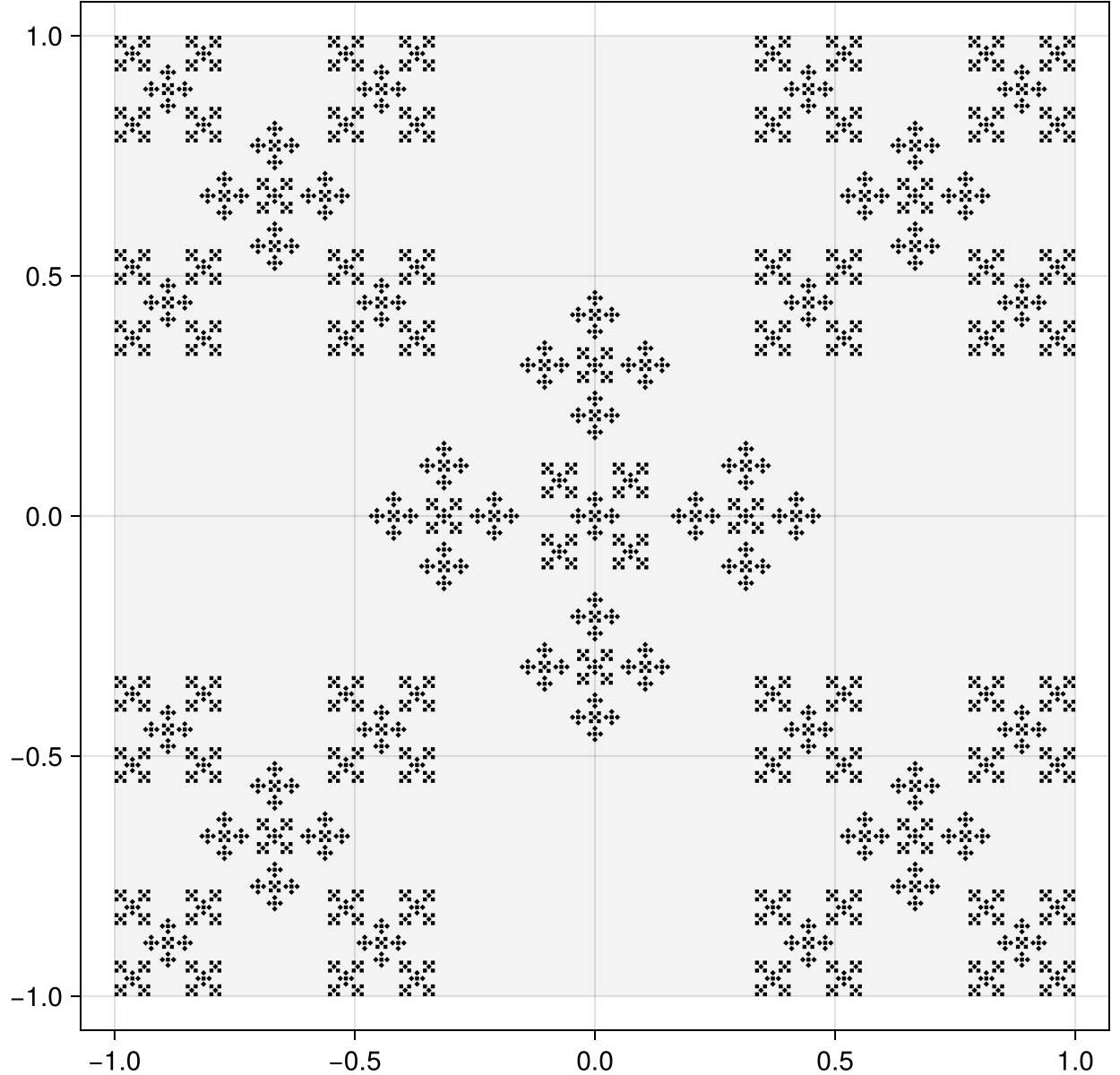}}
	}
	\caption{Three Vicsek prefractals.}\label{fig:vicsek}
\end{figure}
Remark that for \( \theta\in \{0, \frac{\pi}{2}\} \), the polynomial space \( \spPolyQ{N} \) is \( \IFS \)-invariant, while this is not the case for other values of \( \theta \in (0, \frac{\pi}{2}) \).

In some experiments, we worked with the 1D Cantor set, where we used the Hausdorff measure, see \cref{rem:hausdorff}.

\subsection{Behavior of cubature weights}\label{sec:numeric_weights}

In our first experiment, we compute cubature weights for two examples of fractals: Cantor set (\( S_1(x) = \rho x \) and \( S_2(x) = \rho x + (1-\rho) \)) and Vicsek fractal with \( \theta = 0 \), see \cref{fig:points-weights}.
\begin{figure}[hbtp]
	\centering
	\subfloat[Cantor set with \( \rho=1/3 \), \( M = 255 \)]{\includegraphics[width=0.4\textwidth]{{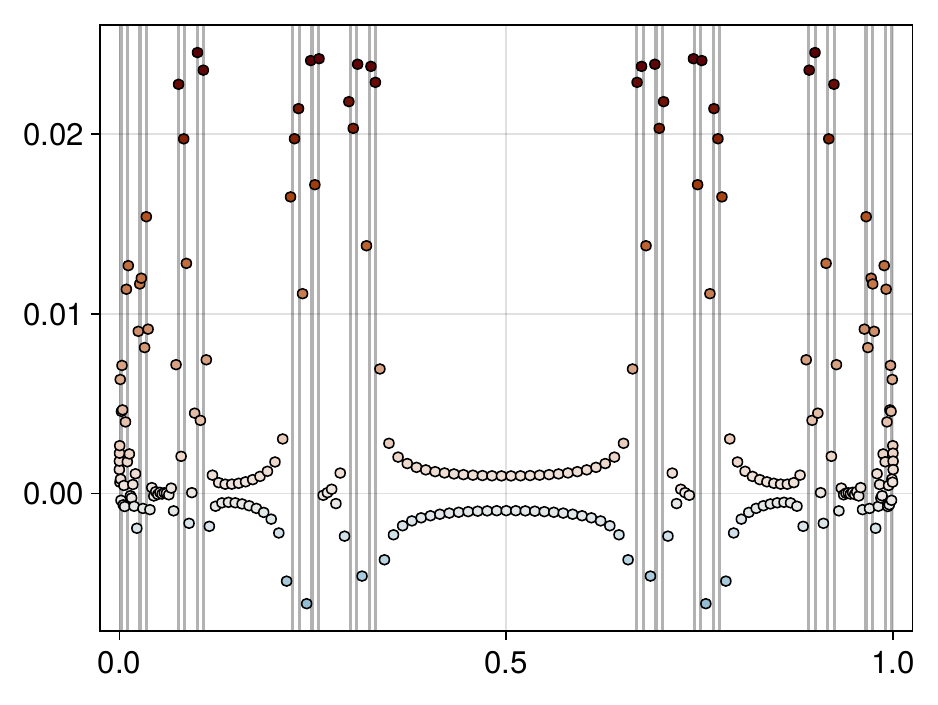}}
	}
	\hspace{2em}
	\subfloat[Vicsek with \( \theta=0.4 \),  \( M = 100 \)]{\includegraphics[width=0.4\textwidth]{{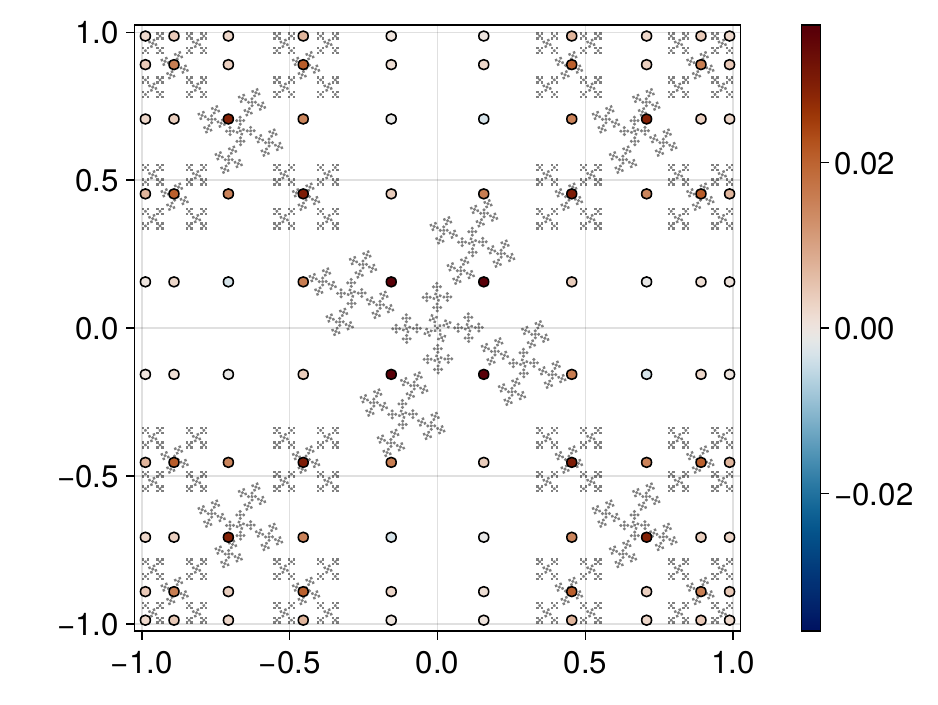}}
	}
	\caption{Left figure: the values of weights (\( y \)-axis) vs the quadrature points (\( x \)-axis). The intersection of vertical lines with \( y=0 \) shows an approximate location of the Cantor set. Right figure: location of cubature points in the 2D-set \( K \) used to construct a cubature on the Vicsek fractal. Values of corresponding quadrature/cubature weights are indicated in color.}\label{fig:points-weights}
\end{figure}
First, the weights can be negative. Second, they seem to be larger in modulus for cubature points close to \( \Gamma \) as one expects intuitively.

To illustrate that negative weights do not seem to affect the cubature accuracy for large \( M \), see \cref{eq:error_cubature_recall}, we plot \( \abs{\vec{w}^M}_1 \) in \cref{fig:growth_weights} as a function of \( M \).
\begin{figure}[hbtp]
	\centering
	\subfloat[Cantor set]{\includegraphics[width=0.4\textwidth]{{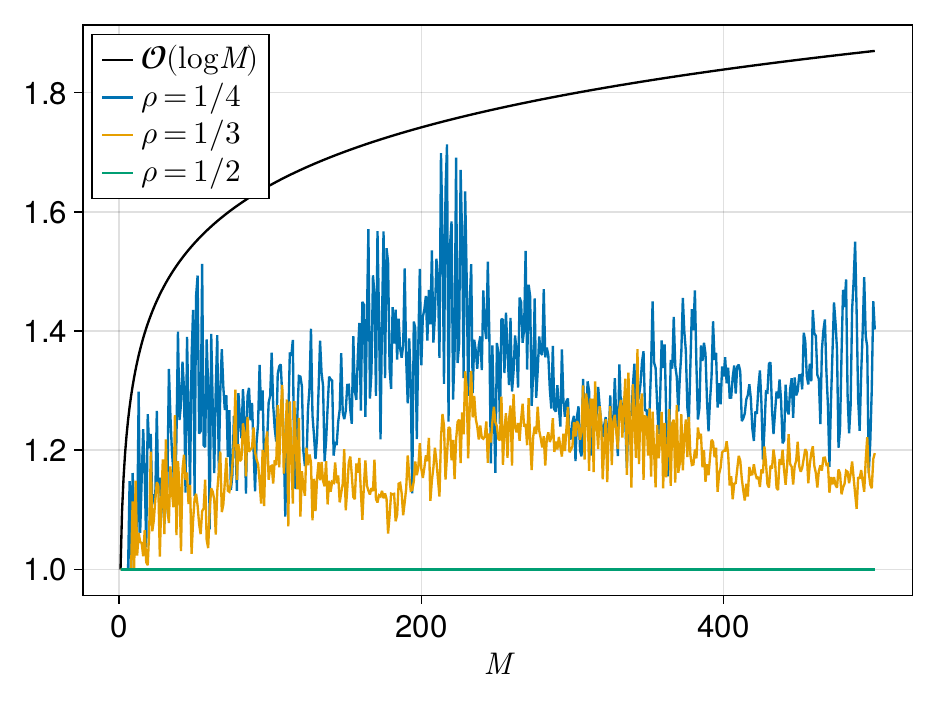}}
	}
	\hspace{2em}
	\subfloat[Vicsek with \( \rho = 1/3 \)]{\includegraphics[width=0.4\textwidth]{{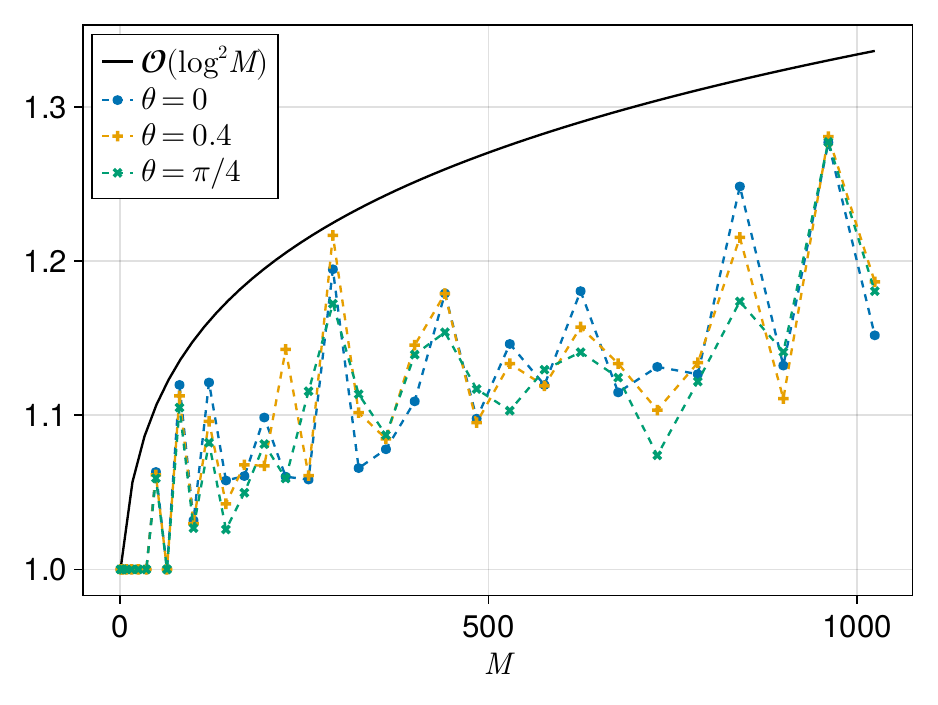}}
	}
	\caption{The dependence of the norm \( |\vec{w}^{M}|_1 \) on \( M \).}\label{fig:growth_weights}
\end{figure}

\subsection{Convergence of the \texorpdfstring{\( h \)}{h}-version and \texorpdfstring{\( p \)}{p}-version}\label{sec:h_conv}

We illustrate in \cref{fig:vicsek-S-inv} the statements of \cref{thm:h-version,thm:error_analytic} by evaluating \( Q[f] \) for
\begin{equation}\label{functionf}
	f(x) = \ex^{\im \kappa \abs{x-x_0}_2} / \abs{x-x_0}_2,  \text{ for }  \kappa=5  \text{ and } x_0 = (0.1, -2),
\end{equation}
which is used in wave scattering applications.
Remark that the location of the singularity \( x_0 \) is chosen outside \( K \).
As expected, \emph{cf.}~\cref{thm:h-version}, since the cubature is exact in both \( \IFS \)-invariant and non-\( \IFS \)-invariant cases.
The results given by \( h \)-version (\cref{fig:vicsek-S-inv-b}) do not differ significantly in these two cases, note that the \( \spPolyQ{k} \) curves overlap.
For the \( p \)-version (\cref{fig:vicsek-S-inv-a}) we also observe that the convergence rates for non-\( \IFS \)-invariant case seem to be the same as in the \( \IFS \)-invariant case.
This seems to confirm that our estimate \cref{esti_pversion_Snoninvariant} in \cref{thm:error_analytic} is not optimal (see \cref{rem_optimalite}).

Remark that to get the \( h^2 \) convergence order, we need \( 4 \) cubature points for \( \Gamma \subset \bbR^2 \), whereas, in~\cite{gibbs2023}, the authors succeeded to reach the \( h^2 \) convergence order with only \( 1 \) point by clever choice of the location of this point.
However, to our knowledge, one does not know how to achieve such a super-convergence for higher orders.

Also, the error for the \(p\)-version stagnates at the order \(\OO(10^{-13})-\OO(10^{-14})\).
This comes from the fact that we compute the cubature weights with the error \(\OO(10^{-14})\), on one hand, and, on the other hand, due to the accumulation of the round off errors.
Indeed, to obtain weights of the cubature for the \(p\)-version, we work with matrices of the size \(N^2 \times N^2\), while for the \(h\)-version their size is independent of \(h\).

\begin{figure}[hbtp]
	\centering
	\subfloat[\( p \)-version\label{fig:vicsek-S-inv-a}]{\includegraphics[width=0.49\textwidth]{{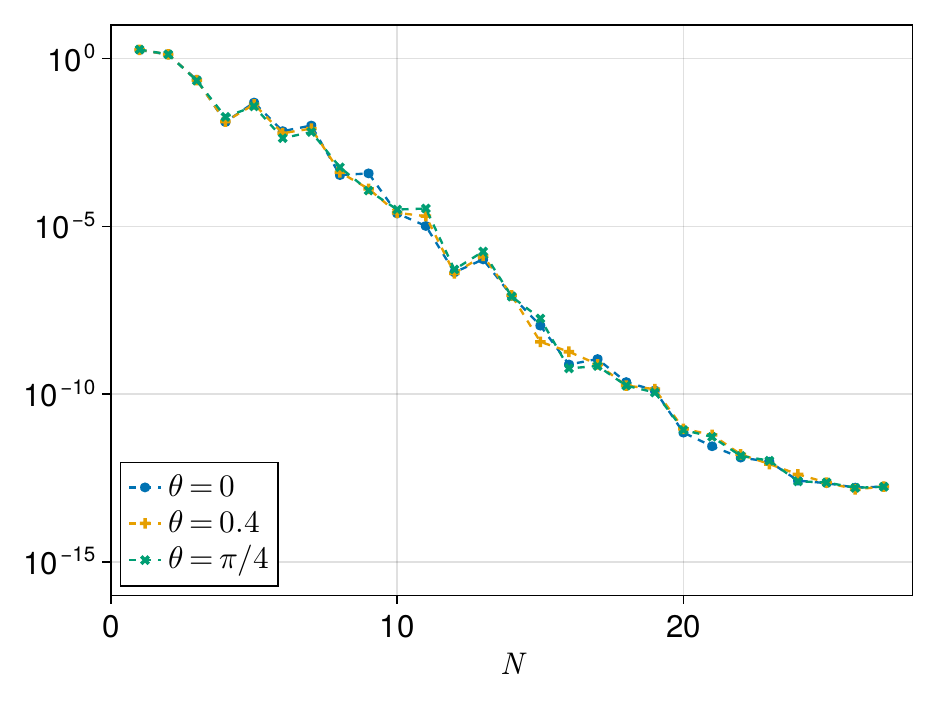}}
	}
	\subfloat[\( h \)-version\label{fig:vicsek-S-inv-b}]{\includegraphics[width=0.49\textwidth]{{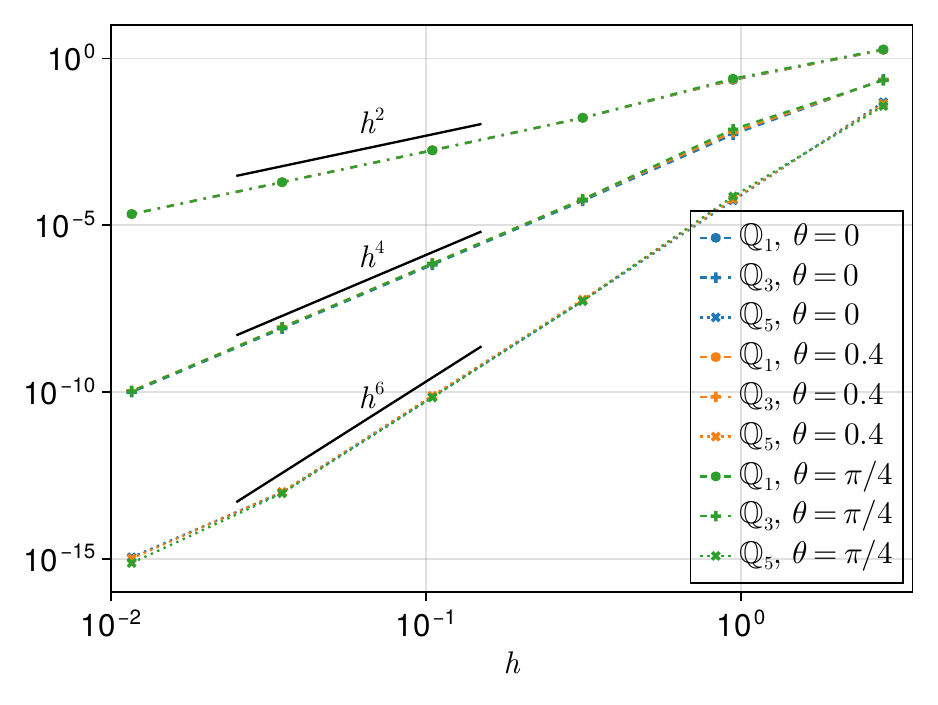}}
	}
	\caption{Convergence for Vicsek without and with rotation and \( \rho = 1/3 \).
		Remark that at this scale the \( \spPolyQ{k} \) curves for different angles \( \theta \) are almost indistinguishable.}\label{fig:vicsek-S-inv}
\end{figure}

\subsection{Integrating almost singular functions}\label{sec:intGamma}

In this section, we perform the experiments for a 2D Cantor dust (\(S_{i,j}(x,y)=S_i(x)S_j(y)\), \(i, j=1, 2\), with  \(S_1(x)=x/3-2/3\) and \(S_2(x)=x/3+2/3\)), located inside \( K = \icc{-1}{1}^2 \).
We work with the  Hausdorff measure, \emph{cf.} \cref{rem:hausdorff}.
We compute integrals of the function \( f \) for \( \kappa=5 \), see \cref{functionf}, with different choices of \( y \) in \( x_0 = (0.1, y) \).
For \( y\in (-2,\,-1) \), \( x_0 \) is located outside \( K \) and when \( y \) approaches \( -1 \) from below, \( x_0 \) approaches \( K \), meaning that the integrand \( f \) becomes almost singular.
For \( y=-1 \), \( x_0 \) lies on the boundary of \( K \), and for \( y=0 \), \( x_0 \) is inside \( K \).

When analyzing the error of the \( h \)-version of the cubature, we remark that for the two latter cases, \( f \) is not continuous in \( K \), but for \( h \) small enough, \( f \) is smooth in the domain \( K_h \supset \Gamma \) defined in \cref{defKh}.
We expect the cubature \cref{eq:new_cub} to converge at the maximal rate, \emph{cf.}~\cref{rem:Kh}.
This is confirmed by the numerical results in \cref{fig:hp_sing}, right.
On the other hand, the convergence analysis of the \( p \)-version relies on the smoothness of \( f \) inside \( K \) (and its analyticity properties), \emph{cf.}~\cref{thm:approx_error2,thm:error_analytic}.
Thus, we expect the convergence to deteriorate as \( x_0 \) approaches \( K \).
This is confirmed by numerical results in \cref{fig:hp_sing}, left, where the worst results are observed for \( y \geq -1.25 \).
\begin{figure}[hbtp]
	\centering
	\subfloat[\( p \)-version]{\includegraphics[width=0.49\textwidth]{{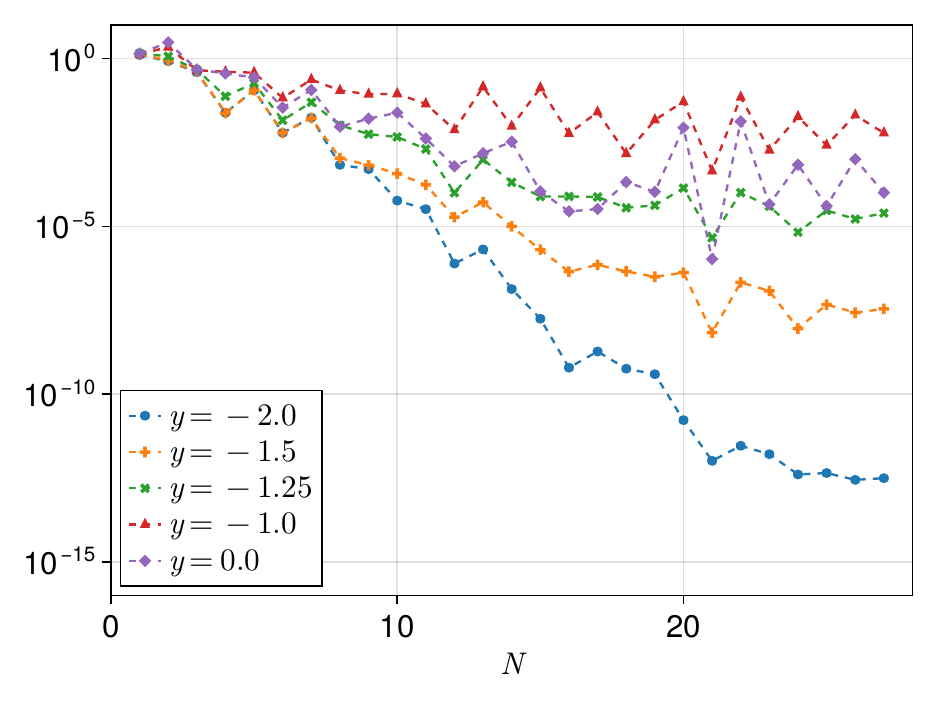}}
	}
	\subfloat[\( h \)-version]{\includegraphics[width=0.49\textwidth]{{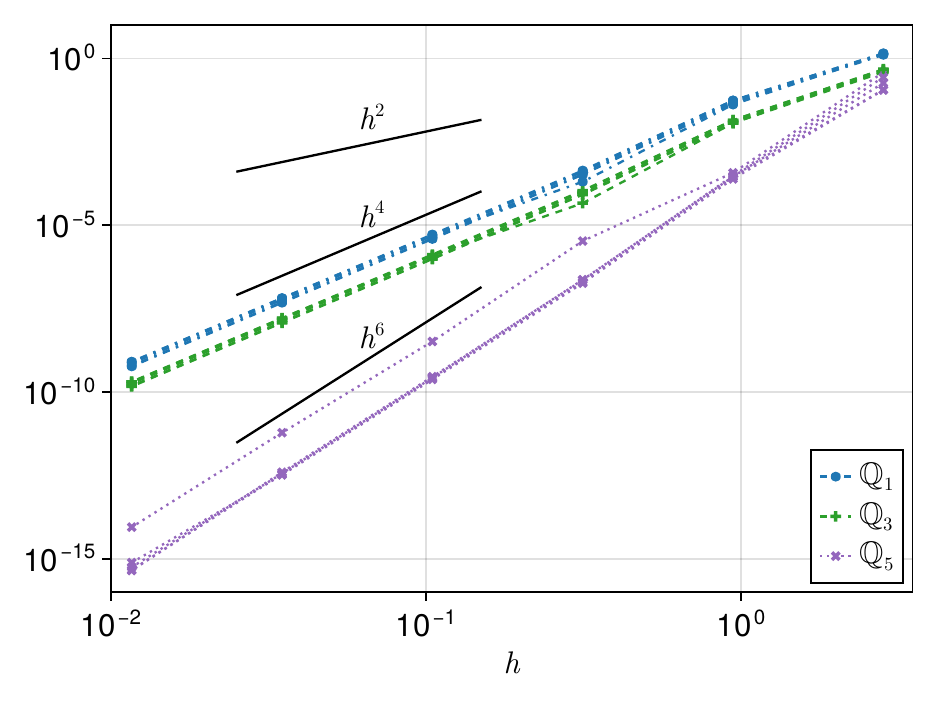}}
	}
	\caption{Comparison of the \( p \)- and \( h \)-versions. In the right plot, we mark with the same colors the curves corresponding to different \( x_0 \) but the same value of \( p \) in the space \( \spPolyQ{p} \).}\label{fig:hp_sing}
\end{figure}

\subsection{Integrating over more exotic attractors of IFS}\label{sec:exotic-attractors}

This section is dedicated to computation of integrals on less classic fractals.
We approximate \( \int_{\Gamma} f \di{\mu} \) for \( f \) defined in \cref{functionf}, over four different self-affine sets: the fat Sierpi\'nski triangle, the Barnsley fern, the Koch snowflake, as well as a non-symmetric Cantor dust, see \cref{ex:sas}.
The fat Sierpi\'nski triangle does not satisfy an open set condition. The IFS of the Barnsley fern is not \emph{similar}, and, moreover, for one of the matrices \( A_{\ell} \), \( \Ker A_{\ell}\neq \{0\} \).
The IFS for the Koch snowflake and the non-symmetric Cantor dust contain transformations with non-trivial rotations; moreover, the Hausdorff dimension of the Koch snowflake is \( d=2 \).
The invariant measures for these examples were chosen as follows: for the fat Sierpi\'nski triangle, \( \mu_{\ell}=1/3 \), \( \ell\in\bbL \); for the Koch snowflake we choose \( \mu_{\ell}=\rho_{\ell}^2 \), \( \ell\in\bbL \); for the Barnsley fern \( \mu_\ell \)'s have been chosen according to~\cite[Tbl.~3.8.3]{barnsley}; finally, for the non-symmetric Cantor dust, we again choose \( \mu_{\ell}=\rho_{\ell}^d \), \( \ell\in\bbL \), with \( d \) solving \( \sum_{\ell\in \bbL}\rho_{\ell}^d=1 \), \emph{cf.}~\cref{rem:hausdorff}.

The results for the \( h \)- and \( p \)- versions of the cubature for the Barnsley fern are shown in \cref{fig:barnsley}, and for the rest of the fractals in \cref{fig:remaining}.
We see that the \( h \)-version converges quite neatly for all the fractals in question.
At the same time, the \( p \)-version performs slightly worse for the non-symmetric Cantor dust and the Barnsley fern.
We do not know a precise reason for this (perhaps this is related to the choice of points on a Cartesian grid, which is probably non-optimal in this case).

\begin{figure}[hbtp]
	\centering
	\subfloat[\( p \)-version]{\includegraphics[width=0.49\textwidth]{{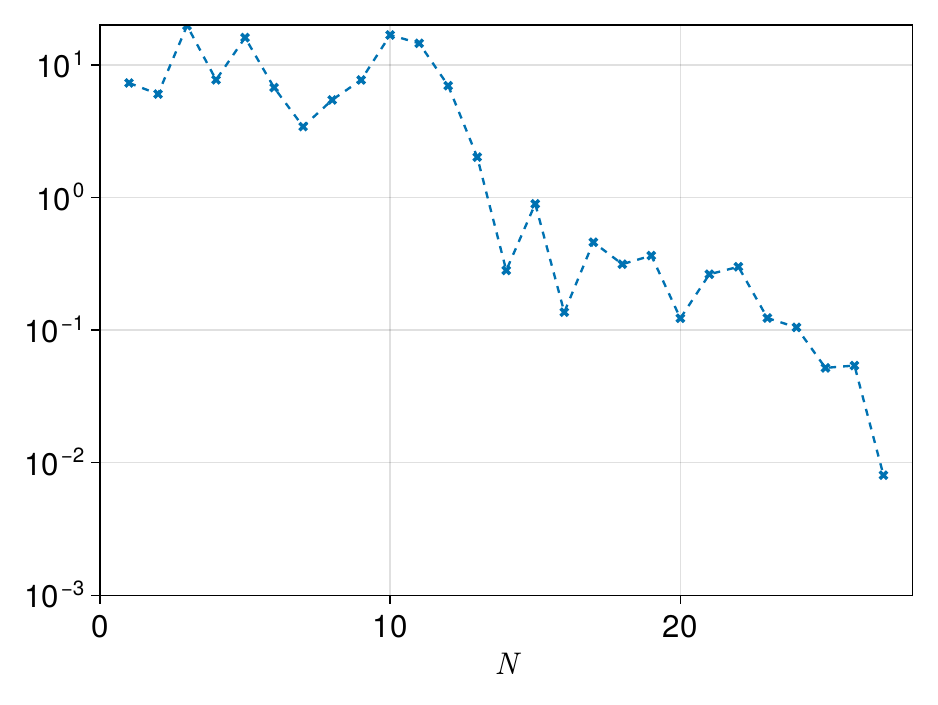}}
	}
	\subfloat[\( h \)-version]{\includegraphics[width=0.49\textwidth]{{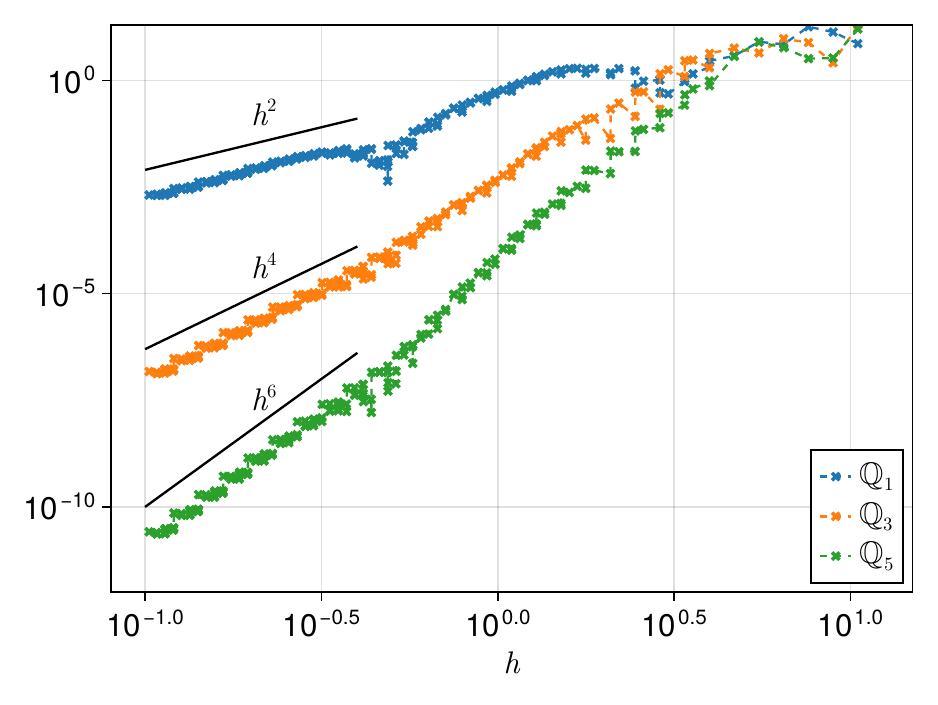}}
	}
	\caption{Convergence for the Barnsley fern. Note the difference of the scales in both figures.}\label{fig:barnsley}
\end{figure}
\begin{figure}[hbtp]
	\centering
	\subfloat[\( p \)-version]{\includegraphics[width=0.49\textwidth]{{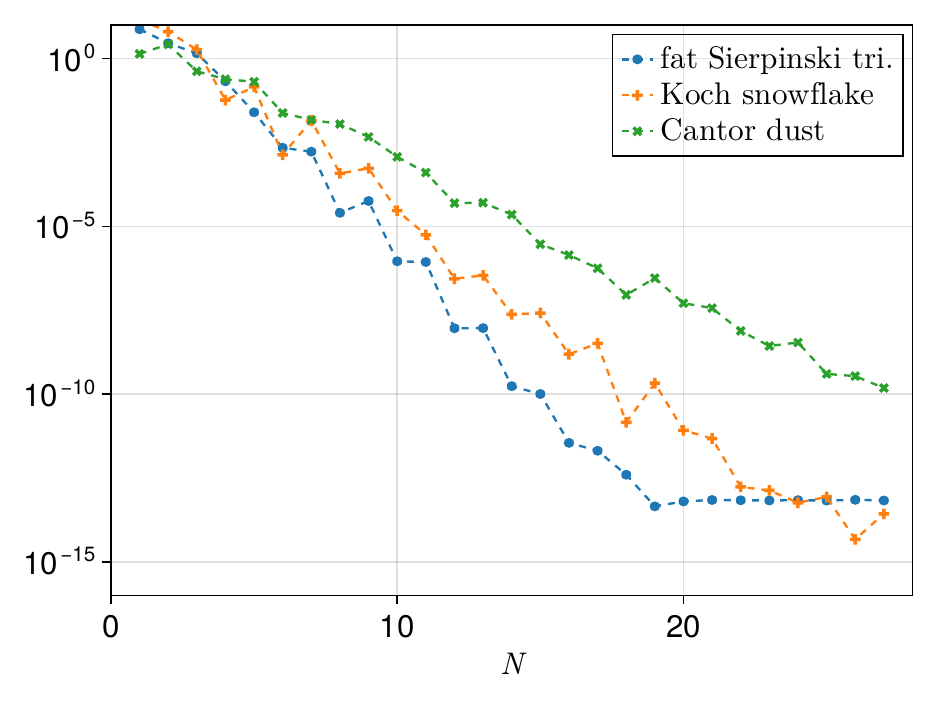}}
	}
	\subfloat[\( h \)-version]{\includegraphics[width=0.49\textwidth]{{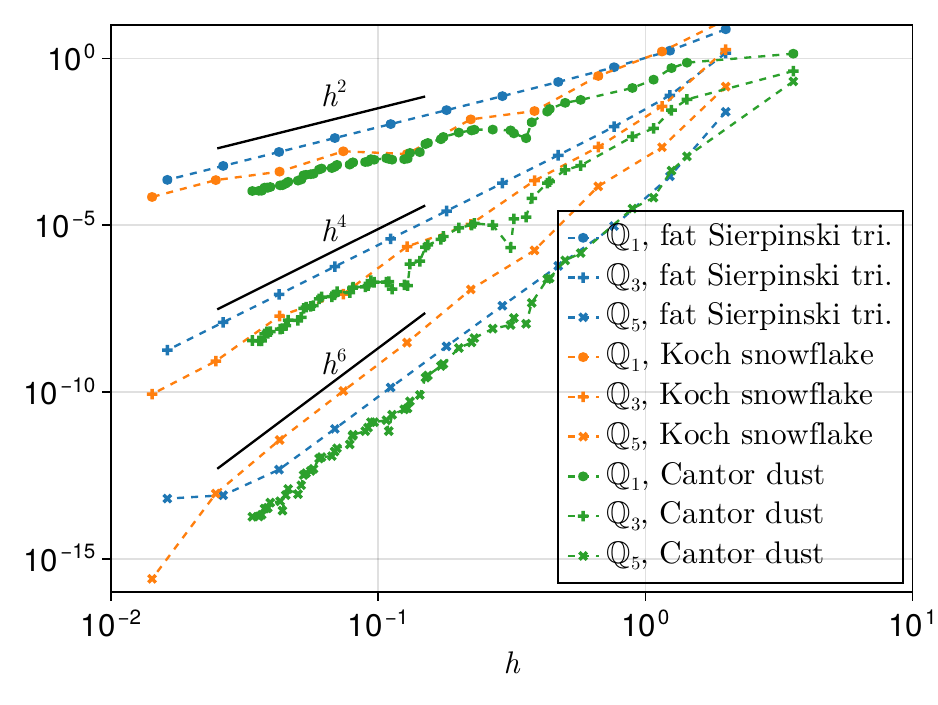}}
	}
	\caption{Convergence for the fat Sierpi\'nski gasket, Koch snowflake and Cantor dust.}\label{fig:remaining}
\end{figure}

\section*{Acknowledgements}
We are grateful to David Hewett (UCL) and Andrew Gibbs (UCL) for fruitful discussions.
Additionally, we thank Andrew Gibbs and Simon Chandler-Wilde (U. Reading) for suggesting the numerical experiment in \cref{sec:intGamma}.
We are deeply grateful to anonymous reviewers for their helpful comments and remarks.
We acknowledge the use of ChatGPT to correct occasional grammar and style errors (in particular, in the introduction).

\section*{Funding}
This research is funded by the project RAPID HyBOX\@.

\printbibliography

\appendix

\section{An extension of the Strichartz algorithm for polynomial integration}\label{appendix:gen}
In this section we generalize the method of~\cite{Str2000} for computing integrals of polynomials on self-affine sets, defined by the IFS \( \IFS = \setst{S_\ell}{\ell \in \bbL} \) via \cref{thm:attractor-existence}, with  respect to the invariant measure \( \mu \), defined via \cref{thm:inv_measure}.
Recall the invariance property \cref{pro:invariance-integral}:
\[
	\int_\Gamma f(x) \di{\mu}
	= \sum_{\ell \in \bbL} \mu_\ell \int_\Gamma f \circ S_\ell(x) \di{\mu}
	= \int_\Gamma \OpFct\clr{f}(x) \di{\mu},
	\qquad \forall  f \in \spC^0(\Gamma).
\]
We want to compute the moments \( m_{\vec{\alpha}} \) defined as the integrals of the monomials \( x^{\vec{\alpha}} \).
This is done by induction on the total degree.
Let \( k \geq 1 \), assume that  \( m_{\vec{\beta}} \) is known for \( \abs{\vec{\beta}}_1 < k \). For \( \abs{\vec{\alpha}}_1 = k \), to the triangular structure of \( \OpFct \), see \cref{eq:opfct-decomposition}, we have
\[
	m_{\vec{\alpha}}
	= \int_\Gamma x^{\vec{\alpha}} \di{\mu}
	= \int_\Gamma \OpFct\clr{x^{\vec{\alpha}} } \di{\mu}
	= \int_\Gamma \OpFct_{k,k}\clr{x^{\vec{\alpha}} } \di{\mu}
	+ \sum_{k' < k} \int_\Gamma \OpFct_{k,k'}\clr{x^{\vec{\alpha}} } \di{\mu}.
\]
This can be rewritten as a linear system with unknown vector  \( M_k =  \plr{m_{\vec{\alpha}}}_{\abs{\vec{\alpha}} = k} \) of the form  \( \plr*{\matI - F_k} M_k = R_k \) where \( F_k \) is the matrix of the operator \( \OpFct_{k,k} \) on the monomials' basis \( \plr{x^{\vec{\alpha}}}_{\abs{\vec{\alpha}} = k} \) and the right-hand side \( R_k \) is known since  \( \OpFct_{k,k'}\clr{x^{\vec{\alpha}} } \) has total degree \( k'< k \).
\Cref{lem:loc-spectrum} implies that \( 1 \) is not an eigenvalue of \( F_k \) and therefore the matrix \( \matI - F_k \) is invertible, and \( U_k = \plr{\matI - F_k}^{-1} R_k \).\footnote{This method has been implemented, see folder \texttt{symbolic-polynomial-integration} in~\cite{ZM_code}, using \texttt{Python} and \texttt{SymPy}~\cite{SymPy}.}

\end{document}